\documentclass[10pt]{article}
\usepackage{cite}
\usepackage{amsmath,amssymb,amsthm}
\usepackage{titlesec,hyperref}
\usepackage{color}
\usepackage{fancyhdr}
\usepackage[margin=3cm]{geometry}
\usepackage{ulem}
\usepackage{amsmath,amssymb,amsthm,mathrsfs,dsfont}
\usepackage[margin=3cm]{geometry}
\usepackage[integrals]{wasysym}
\titleformat{\subsection}{\it}{\thesubsection.\enspace}{1pt}{}

\makeatletter
\def\ps@pprintTitle{%
   \let\@oddhead\@empty
   \let\@evenhead\@empty
   \let\@oddfoot\@empty
  \let\@evenfoot\@oddfoot
}
\makeatother

\newtheorem{theo}{Theorem}[section]
\newtheorem{lemm}[theo]{Lemma}
\newtheorem{defi}[theo]{Definition}
\newtheorem{coro}[theo]{Corollary}
\newtheorem{prop}[theo]{Proposition}
\newtheorem{rema}[theo]{Remark}
\numberwithin{equation}{section}

\def\ud{{\rm d}}

\def\bel{\begin{equation}\label}

\def\eeq{\end{equation}}

\newcommand{\beq}{\begin{equation}}
\newcommand{\beno}{\begin{equation*}}
\newcommand{\eeno}{\end{equation*}}

\allowdisplaybreaks
\begin{document}
\title{Global well-posedness for the non-viscous MHD equations with magnetic diffusion in critical Besov spaces}
\author{
Weikui $\mbox{Ye}^1$\footnote{email: 904817751@qq.com}
\quad and\quad
 Zhaoyang $\mbox{Yin}^{1,2}$\footnote{email: mcsyzy@mail.sysu.edu.cn}\\
 $^1\mbox{Department}$ of Mathematics,
Sun Yat-sen University,\\ Guangzhou, 510275, China\\
$^2\mbox{Faculty}$ of Information Technology,\\ Macau University of Science and Technology, Macau, China}
\date{}
\maketitle
\begin{abstract}
 In this paper, we mainly investigate the Cauchy problem of the non-viscous MHD equations with magnetic diffusion. We first establish the local well-posedness (existence,~uniqueness and continuous dependence) with initial data $(u_0,b_0)$ in critical Besov spaces $ {B}^{\frac{d}{p}+1}_{p,1}\times{B}^{\frac{d}{p}}_{p,1}$ with $1\leq p\leq\infty$, and give a lifespan $T$ of the solution which depends on the norm of the Littlewood-Paley decomposition
of the initial data. Then, we prove the global existence in critical Besov spaces. In particular, the results of global existence also hold in Sobolev space $ C([0,\infty); {H}^{s}(\mathbb{S}^2))\times \Big(C([0,\infty);{H}^{s-1}(\mathbb{S}^2))\cap L^2\big([0,\infty);{H}^{s}(\mathbb{S}^2)\big)\Big)$ with $s>2$, when the initial data satisfies $\int_{\mathbb{S}^2}b_0dx=0$ and $\|u_0\|_{B^1_{\infty,1}(\mathbb{S}^2)}+\|b_0\|_{{B}^{0}_{\infty,1}(\mathbb{S}^2)}\leq \epsilon$. It's worth noting that our results imply some large and low regularity initial data for the global existence, which improves considerably the recent results in \cite{weishen}.
\end{abstract}


\tableofcontents
\section{Introduction}
The incompressible magnetohydrodynamic (MHD) equations can be written as follows:
\begin{equation}\label{vu}
\left\{\begin{array}{lll}
u_t-\mu\Delta u+\nabla P =b\nabla b-u\nabla u,\\
b_t-\nu\Delta b+u\nabla b=b\nabla u ,\\
{\rm{div}}u={\rm{div}}b=0, \\
(u,b)|_{t=0}=(u_0,b_0),
\end{array}\right.
\end{equation}
where the vector fields $u=(u_1,u_2,...,u_d),\ b=(b_1,b_2,...,b_d)$ are the velocity and magnetic respectively, the scalar function $P$ denotes the pressure. The MHD equation is a coupled system of the Navier-Stokes equation and Maxwell's equation. This model describes the interactions between the magnetic field and the fluid of moving electrically charged particles such as plasmas, liquid metals, and electrolytes. For more physical background, one can refer to \cite{Biskamp,Davidson}.

The MHD equations are of great interest in mathematics and physics. Let's review some well-posed results about the MHD equations. When $\nu=0,~\mu\neq 0$ (non-magnetic diffusion but full viscosity) in system \eqref{vu}, we refer to \cite{Abidi,Jiu,Lin-FH,Ren,Pan} about the global existence results with initial data sufficiently close to the equilibrium. The $L^2$ decay rate was studied by Agapito and Schonbek\cite{Agapito-Schonbek}.
Fefferman et al. obtained a local existence result in $\mathbb{R}^d(d=2,3)$ with the initial data $(u_0, B_0)\in H^s(\mathbb{R}^d)\times H^s(\mathbb{R}^d)( s>d/2)$ in \cite{Feffer1} and
$(u_0, B_0)\in H^{s-1-\epsilon}(\mathbb{R}^d)\times H^s(\mathbb{R}^d)(s>d/2,0<\epsilon <1)$ in \cite{Feffer2}. Chemin et al. \cite{chemin1}
improved Fefferman et al.'s results to the inhomogenous Besov space with the initial data
$(u_0, B_0)\in B^{\frac{d}{2}-1}_{2,1}(\mathbb{R}^d)\times B^{\frac{d}{2}}_{2,1}(\mathbb{R}^d) (d=2,3)$ and also proved the uniqueness with $d=3$. Wan \cite{Wan} obtained the uniqueness with $d=2$. Recently, Li, Tan and Yin \cite{Li1} obtained the existence and uniqueness of solutions to \eqref{sp00} with the initial data
$(u_0, B_0)\in \dot{B}^{\frac{d}{p}-1}_{p,1}(\mathbb{R}^d)\times \dot{B}^{\frac{d}{p}}_{p,1}(\mathbb{R}^d)$ $(1\leq p\leq 2d)$.

However, when $\mu=0,~\nu\neq 0$ (non-viscosity but full magnetic diffusion), \eqref{vu} becomes
\begin{equation}\label{sp00}
\left\{\begin{array}{lll}
u_t+\nabla P =b\nabla b-u\nabla u,\\
b_t-\Delta b+u\nabla b=b\nabla u ,\\
{\rm{div}}u={\rm{div}}b=0, \\
(u,b)|_{t=0}=(u_0,b_0).
\end{array}\right.
\end{equation}
The study of \eqref{sp00} will become more difficult, especially the local well-posedness in the critical Besov space is unknown.
Recently, Wei and Zhang \cite{weishen} proved the global existence with sufficient small initial data in $H^4(\mathbb{S}^2)\times H^4(\mathbb{S}^2)$. This is a new result for the global well-posedness, which does not require the initial data to be near the equilibrium. Hassainia \cite{Hassainia} studied the global well-posedness for the three-dimensional MHD equations with axisymmetric initial data. 

In \cite{weishen}, Wei and Zhang proposed that it's an open problem to study the global well-posedness with low regularity of the initial data. In this paper, we invesigate the local well-posedness for \eqref{sp00} in the critical Besov space ${B}^{\frac{d}{p}+1}_{p,1}(\mathbb{T}^d)\times \big({B}^{\frac{d}{p}}_{p,1}(\mathbb{T}^d)\cap{B}^{\frac{d}{p}+2}_{p,1}(\mathbb{T}^d)\big),~\mathbb{T}=\mathbb{R}~or~\mathbb{S}$. In the periodic case ($d=2$), we solve this problem especially for the initial data with low regularity(see Theorem \ref{theoremglobal1} below).

The main difficulty is that the system is only partially parabolic, owing to the Euler type equation which is of hyperbolic type (when $b=0$,~\eqref{sp00} is the Euler equation). Hence, the continuous dependence and the global existence of the solutions are hard to deal with. However, in this paper, we will adopt the methods such as frequency decomposition and decomposition of equations to obtain the continuous dependence for \eqref{sp00}, which can be applied to prove the continuous dependence for general Euler type equations. For the global existence with $d=2$, the term $b{\rm div}b$ is hard to estimate. Fortunately, we find the $L^2$ decay of $b(t,x)$. Combining the $L^2$ decay and the critical estimations for the transport equation in the Besov spaces (see Lemma \ref{priori estimate0}), we can get the global existence of \eqref{sp00}. It's worth noting that we obtain the $L^2$ decay of $b(t,x)$ by setting $\|u_0\|_{B^1_{\infty,1}}+\|b_0\|_{{B}^{0}_{\infty,1}}$ small enough, not the $\|u_0\|_{H^4},~\|b_0\|_{H^4}$, which means some results in \cite{weishen} can be improved and some large initial data will make sense.

Our main theorem can be stated as follows.
\begin{theo}\label{theorem}
Let $(u_0,b_0)\in \dot{B}^{\frac{d}{p}+1}_{p,1}(\mathbb{T}^d)\times \dot{B}^{\frac{d}{p}}_{p,1}(\mathbb{T}^d) $ with $d\geq 2,~p\in[1,\infty]$. Then there exists a positive time $T$ such that \eqref{sp00} is locally well-posed in $E^p_T$ in the sense of Hadamard, where $E^p_T:=C([0,T];\dot{B}^{\frac{d}{p}+1}_{p,1}(\mathbb{T}^d))\times \Big(C([0,T];\dot{B}^{\frac{d}{p}}_{p,1}(\mathbb{T}^d))\cap L^1([0,T];\dot{B}^{\frac{d}{p}+2}_{p,1}(\mathbb{T}^d))\Big)$.
\end{theo}
\begin{rema}
In \cite{Bourgain1,Bourgain2}, Bourgain and Li employed a combination of Lagrangian and Eulerian techniques to obtain strong local ill-posed results of the Euler
equation in $B^{\frac{d}{p}+1}_{p,r}$ with $p\in[1,\infty),~r\in(1,\infty],~d=2,3$. Recently, Guo, Li and Yin \cite{Lijfa} proved the Euler
equation is well-posed in $B^{\frac{d}{p}+1}_{p,1}$ with $p\in[1,\infty]$, which means that $B^{\frac{d}{p}+1}_{p,1}$ may be the critical Besov space for the well-posedness of the Euler equation. Thus, since \eqref{sp00} is the Euler equation when $b=0$, we conclude that $C([0,T];{B}^{\frac{d}{p}+1}_{p,1}(\mathbb{R}^d))\times \Big(C([0,T];{B}^{\frac{d}{p}}_{p,1}(\mathbb{R}^d))\cap L^1([0,T];{B}^{\frac{d}{p}+2}_{p,1}(\mathbb{R}^d))\Big)$ is also the critical Besov space for \eqref{sp00} with $d=2,3,~p\in[1,\infty]$.
\end{rema}

In the whole space, the global existence for \eqref{sp00} is still an open problem. But in the period case $\mathbb{S}^2$, Wei and Zhang \cite{weishen} presented the following Theorem.
\begin{theo}\label{theoremglobal0}\cite{weishen}
Let $(u_0,b_0)\in  H^4(\mathbb{S}^2)\times H^4(\mathbb{S}^2) $ with ${\rm{div}}u_0={\rm{div}}_0=0$. If
$$\int_{\mathbb{S}^2}b_0{\ud}x=0\quad and\quad \|b_0\|_{H^4(\mathbb{S}^2)}+\|u_0\|_{H^4(\mathbb{S}^2)}\leq \epsilon,\quad \text{for any $\epsilon$ small enough}$$
then \eqref{sp00} has a unique global solution $(u,b)$ belonging to $ C([0,\infty);H^4(\mathbb{S}^2))\times \Big(C([0,\infty);H^4(\mathbb{S}^2))\cap L^2\big([0,\infty);H^5(\mathbb{S}^2)\big)\Big)$.
\end{theo}

%Let $C_{E_0}:=C(1+\|u_0\|_{B^{1+\frac{d}{p}}_{p,1}}+\|b_0\|_{B^{\frac{d}{p}}_{p,1}})$ and $B:=\|b_0\|_{B^{0}_{\infty,1}}\leq C_{E_0}$.

Our obtained results generalize and improve Theorem \ref{theoremglobal0} considerably in the critical Besov spaces, as follows:
\begin{theo}\label{theoremglobal1}
Let $(u_0,b_0)\in  {B}^{s}_{p,1}(\mathbb{S}^2)\times {B}^{s-1}_{p,1}(\mathbb{S}^2)$ with $s\geq 1+\frac{2}{p},~1\leq p\leq \infty$ and ${\rm{div}}u_0={\rm{div}}b_0=0$. If
$$\int_{\mathbb{S}^2}b_0{\ud}x=0\quad and\quad\|u_0\|_{B^1_{\infty,1}(\mathbb{S}^2)}+\|b_0\|_{{B}^{0}_{\infty,1}(\mathbb{S}^2)}\leq \epsilon,\quad \text{for any $\epsilon$ small enough,}$$
then \eqref{sp00} has a unique global solution $(u,b)$ belonging to $ C([0,\infty); {B}^{s}_{p,1}(\mathbb{S}^2))\times \Big(C([0,\infty);{B}^{s-1}_{p,1}(\mathbb{S}^2))\cap L^1\big([0,\infty);{B}^{s+1}_{p,1}(\mathbb{S}^2)\big)\Big).$
\end{theo}

\begin{rema}
In \cite{weishen}, Wei and Zhang proposed that it's an interesting problem to study the global well-posedness with lower regularity of the initial data (may be $H^s(s>2)$ is enough), see Theorem \ref{theoremglobal0}. We solve this problem with lower regularity such that $(u_0,b_0)\in  {H}^{s}(\mathbb{S}^2)\times {H}^{s-1}(\mathbb{S}^2)$ with $s>2$ or the critical case $(u_0,b_0)\in  {B}^{2}_{2,1}(\mathbb{S}^2)\times {B}^{1}_{2,1}(\mathbb{S}^2)$, see Theorem \ref{theoremglobal1} with $p=2$.

Moreover, comparing to the result in Theorem \ref{theoremglobal0}, we abate the condition from
$\|u_0\|_{H^4(\mathbb{S}^2)}+\|b_0\|_{H^4(\mathbb{S}^2)}\leq \epsilon$ to $\|u_0\|_{B^1_{\infty,1}(\mathbb{S}^2)}+\|b_0\|_{B^0_{\infty,1}(\mathbb{S}^2)}\leq \epsilon~(p=2)$, which implies the global well-posedness for the some large initial data. For instance, let $s=4$ and $u_0=\frac{1}{10n^{\frac{7}{2}}}\Big(sin(n x_2),sin(n x_1)\Big),~b_0=\frac{1}{10n^{\frac{5}{2}}}\Big(sin(n x_2),sin(n x_1)\Big)$. It's easy to verify that
$${\rm{div}}u_0={\rm{div}}b_0=0,~\int_{\mathbb{S}}b_0{\ud}x=0\quad and\quad \|b_0\|_{{B}^{1}_{2,1}(\mathbb{S}^2)}+\|u_0\|_{B^1_{\infty,1}(\mathbb{S}^2)}\leq\|b_0\|_{H^2(\mathbb{S}^2)}+\|u_0\|_{H^3(\mathbb{S}^2)}\approx\frac{C}{n^{\frac{1}{2}}} .$$
However, we have
$$ \|u_0\|_{H^3(\mathbb{S}^2)},\|b_0\|_{H^2(\mathbb{S}^2)}\approx n^{\frac{1}{2}}.$$
\end{rema}

For non-critical cases, Let ${C}_{E_0}:=C(\|u_0\|_{H^{s}}+\|b_0\|_{H^{s-1}})~(s>2)$. Since the interpolation yields that~($d=2$)
$$\|u_0\|_{B^1_{\infty,1}}\leq \|u_0\|^{\theta}_{L^2}\|u_0\|^{1-\theta}_{H^s},\quad \|b\|_{L^{\infty}}\leq \|b\|^{\bar{\theta}}_{L^2}\|b\|^{1-\bar{\theta}}_{B^{s-2}_{\infty,\infty}},\quad \forall s>2,$$
where $\theta=1-\frac{2}{s}$ and $\bar{\theta}=1-\frac{1}{s-1}$. By Theorem \ref{theoremglobal1}, after some modifications one can obtain the global well-posedness with some weaker conditions in the Sobolev space .
\begin{coro}\label{theoremglobal2}
Let $(u_0,b_0)\in  {H}^{s}(\mathbb{S}^2)\times{H}^{s-1}(\mathbb{S}^2) $ with any $s>2$ and ${\rm{div}}u_0={\rm{div}}b_0=0$. Assume
\begin{align}\label{xiao2}
\int_{\mathbb{S}^2}b_0{\ud}x=0 ~~ and~~\|b_0\|_{L^2(\mathbb{S}^2)}+ \|u_0\|_{L^2(\mathbb{S}^2)}\leq \min\{\frac{1}{8C^2},(\frac{c\bar{\theta}}{{C}_{E_0}+1})^{\frac{1}{{\theta}}},
(\frac{\|b_0\|_{B^0_{\infty,1}}}{{C}_{E_0}})^{\frac{1}{\bar{\theta}}} \}.	
\end{align}
Then \eqref{sp00} has a unique global solution $(u,b)$ belonging to $ C([0,\infty); {H}^{s}(\mathbb{S}^2))\times \Big(C([0,\infty);{H}^{s-1}(\mathbb{S}^2))\cap L^2\big([0,\infty);{H}^{s}(\mathbb{S}^2)\big)\Big).$
\end{coro}
%Indeed, let $u_0=b_0=\epsilon\phi$ with $\phi=(\phi_1(x_2),\phi_2(x_1))$ and  $\phi_i(x)\in S(\mathbb{S}),~i=1,2$. Then set $\int_{\mathbb{S}}\phi(x) dx=0$. It's easy to see that $u_0$ and $b_0 $ satisfy $div\phi=0$ and \eqref{xiao2} for sufficient small $\epsilon$ .

The remainder of the paper is organized as follows. In Section 2, we introduce some useful preliminaries. In Section 3, we establish the local existence and uniqueness of the solution to \eqref{sp00} when the expression of local time be given. In Section 4, we present the data-to-solutions map depends continuously on
the initial data with the common
lifespan. In Section 5, we prove the global existence of \eqref{sp00} with large initial data.\\
\textbf{Notations:} Throughout, we denote $\dot{B}^{s}_{p,r}(\mathbb{T}^d)=\dot{B}^{s}_{p,r}$, $\|u\|_{\dot{B}^{s}_{p,r}(\mathbb{T}^d)}+\|v\|_{\dot{B}^{s}_{p,r}(\mathbb{T}^d)}=\|u,v\|_{\dot{B}^{s}_{p,r}}$ and $C([0,T];\dot{B}^{s}_{p,r}(\mathbb{T}^d))=C_T(\dot{B}^{s}_{p,r})$, $L^p([0,T];\dot{B}^{s}_{p,r}(\mathbb{T}^d))=L^p_T(\dot{B}^{s}_{p,r})$ with $\mathbb{T}=\mathbb{R}~or~\mathbb{S} $.
\section{Preliminaries}
\par
In this section, we will recall some properties about the Littlewood-Paley decomposition and Besov spaces.
\begin{prop}\cite{book}
Let $\mathcal{C}$ be the annulus $\{\xi\in\mathbb{R}^d:\frac 3 4\leq|\xi|\leq\frac 8 3\}$. There exist radial functions $\chi$ and $\varphi$, valued in the interval $[0,1]$, belonging respectively to $\mathcal{D}(B(0,\frac 4 3))$ and $\mathcal{D}(\mathcal{C})$, and such that
$$ \forall\xi\in\mathbb{R}^d,\ \chi(\xi)+\sum_{j\geq 0}\varphi(2^{-j}\xi)=1, $$
$$ \forall\xi\in\mathbb{R}^d\backslash\{0\},\ \sum_{j\in\mathbb{Z}}\varphi(2^{-j}\xi)=1, $$
$$ |j-j'|\geq 2\Rightarrow\mathrm{Supp}\ \varphi(2^{-j}\cdot)\cap \mathrm{Supp}\ \varphi(2^{-j'}\cdot)=\emptyset, $$
$$ j\geq 1\Rightarrow\mathrm{Supp}\ \chi(\cdot)\cap \mathrm{Supp}\ \varphi(2^{-j}\cdot)=\emptyset. $$
The set $\widetilde{\mathcal{C}}=B(0,\frac 2 3)+\mathcal{C}$ is an annulus, and we have
$$ |j-j'|\geq 5\Rightarrow 2^{j}\mathcal{C}\cap 2^{j'}\widetilde{\mathcal{C}}=\emptyset. $$
Further, we have
$$ \forall\xi\in\mathbb{R}^d,\ \frac 1 2\leq\chi^2(\xi)+\sum_{j\geq 0}\varphi^2(2^{-j}\xi)\leq 1, $$
$$ \forall\xi\in\mathbb{R}^d\backslash\{0\},\ \frac 1 2\leq\sum_{j\in\mathbb{Z}}\varphi^2(2^{-j}\xi)\leq 1. $$
\end{prop}

\begin{defi}\cite{book}
Let $u$ be a tempered distribution in $\mathcal{S}'(\mathbb{R}^d)$ and $\mathcal{F}$ be the Fourier transform and $\mathcal{F}^{-1}$ be its inverse. For all $j\in\mathbb{Z}$, define
$$
\Delta_j u=0\,\ \text{if}\,\ j\leq -2,\quad
\Delta_{-1} u=\mathcal{F}^{-1}(\chi\mathcal{F}u),\quad
\Delta_j u=\mathcal{F}^{-1}(\varphi(2^{-j}\cdot)\mathcal{F}u)\,\ \text{if}\,\ j\geq 0,\quad
S_j u=\sum_{j'<j}\Delta_{j'}u.
$$
Then the Littlewood-Paley decomposition is given as follows:
\begin{align*}
u=\sum_{j\in\mathbb{Z}}\Delta_j u \quad \text{in}\ \mathcal{S}'(\mathbb{R}^d).	
\end{align*}
Let $s\in\mathbb{R},\ 1\leq p,r\leq\infty.$ The nonhomogeneous Besov space $B^s_{p,r}(\mathbb{R}^d)$ is defined by
$$ B^s_{p,r}=B^s_{p,r}(\mathbb{R}^d)=\{u\in S'(\mathbb{R}^d):\|u\|_{B^s_{p,r}(\mathbb{R}^d)}=\Big\|(2^{js}\|\Delta_j u\|_{L^p(\mathbb{S}^d)})_j \Big\|_{l^r(\mathbb{Z})}<\infty\}. $$
\end{defi}

%\begin{lemm}\label{compact}
%Let $s\in (-\frac{d}{p'},\frac{d}{p}]$ ($s=\frac{d}{p},r=1$). Assume $f^n$ is uniformly bounded in $\dot{B}^s_{p,r}\cap\dot{B}^{-\delta}_{\infty,\infty}(\forall\delta>0)$ or $\dot{B}^s_{p,r}\cap L^{\infty}$ . Then $\varphi f^n$ is bound in $\dot{B}^{s}_{p,r}\cap\dot{B}^{s-\epsilon_1}_{p,r}$ $(0<\epsilon_1<s+\frac{d}{p'})$, and the map $f^n\mapsto \varphi f^n$ is compact in $\dot{B}^{s-\epsilon}_{p,r}$ $(0<\epsilon<\epsilon_1)$, where $\varphi\in S(\mathbb{R}^d)$.
%\end{lemm}
%\begin{proof}
%The proof is based on Theorems 2.93-2.94 in \cite{book}, we omit it here.
%\end{proof}

%\begin{lemm}\label{zhibiao}\cite{miao}
%Let $s_1,s_2\leq \frac{d}{p}$ and $s_1+s_2>d\max\{0,\frac{2}{p}-1\}$. Assume $f\in\dot{B}^{s_1}_{p,1}$ and $g\in\dot{B}^{s_2}_{p,1}$. Then there holds
%$$\|fg\|_{\dot{B}^{s_1+s_2-\frac{d}{p}}_{p,\infty}}\leq C\|f\|_{\dot{B}^{s_1}_{p,\infty}}\|g\|_{\dot{B}^{s_2}_{p,1}}.$$
%\end{lemm}
For the periodic case, we have the following definition.
\begin{defi}
Let $u\in D'(\mathbb{S}^d)$. We similarly denote $\mathcal{F}$ by the Fourier transform and $\mathcal{F}^{-1}$ by its inverse
\[(\mathcal{F}u(x))(\xi)=\frac{1}{(2\pi)^d}\int_{\mathbb{S}^d}u(x)e^{-i\xi x}{\ud}x,\quad (\mathcal{F}^{-1}v(\xi))(x)=\frac{1}{(2\pi)^d}\sum_{\xi\in \mathbb{Z}^d}v(\xi)e^{i\xi x} \quad \xi\in \mathbb{Z}^d.\]
We define the Littlewood-Paley operators $\Delta_j$ by
$$
\Delta_j u=0\,\ \text{if}\,\ j\leq -2,\quad
\Delta_{-1} u=\int_{\mathbb{S}^d}u(x){\ud}x,\quad
\Delta_j u=\mathcal{F}^{-1}(\varphi(2^{-j}\cdot)\mathcal{F}u)\,\ \text{if}\,\ j\geq 0,\quad
S_j u=\sum_{j'<j}\Delta_{j'}u.
$$
where the functions $\varphi$ is defined in Proposition 2.1.
We can then define the Besov space $B^s_{p,r}(\mathbb{S}^d)$ such that
$$ B^s_{p,r}=B^s_{p,r}(\mathbb{S}^d)=\{u\in D'(\mathbb{S}^d):\|u\|_{B^s_{p,r}(\mathbb{S}^d)}=\|(2^{js}\|\Delta_j u\|_{L^p(\mathbb{S}^d)})_j\|_{l^r}<+\infty\}. $$
\end{defi}

\begin{defi}\cite{book}
Let $s\in\mathbb{R},1\leq p,q,r\leq\infty$ and $T\in (0,\infty].$ The function space $\widetilde{L}^q_T(\dot{B}^{s}_{p,r})$ is defined as the set of all the distributions satisfying
$\|f\|_{\widetilde{L}^q_T(\dot{B}^{s}_{p,r})}:=\|(2^{ks}\|\dot{\Delta}_kf(t)\|_{L^q_TL^p})_k\|_{l^r}<\infty .$
\end{defi}
Thanks to Minkowski's inequality, it is easy to find that
$$\|f\|_{\widetilde{L}^q_T(\dot{B}^{s}_{p,r})}\leq \|f\|_{L^q_T(\dot{B}^{s}_{p,r})},\quad q\leq r,\quad\quad\quad \|f\|_{\widetilde{L}^q_T(\dot{B}^{s}_{p,r})}\geq \|f\|_{L^q_T(\dot{B}^{s}_{p,r})},\quad q\geq r.$$

Finally, we intruduce some useful results about the following heat conductive equation and the transport equation
\begin{equation}\label{s1cuchong}
\left\{\begin{array}{l}
    u_t-\Delta u=G,\ x\in\mathbb{R}^d,\ t>0, \\
    u(0,x)=u_0(x),\ x\in\mathbb{R}^d,
\end{array}\right.
\end{equation}
\begin{equation}\label{s1}
\left\{\begin{array}{l}
    f_t+v\cdot\nabla f=g,\ x\in\mathbb{R}^d,\ t>0, \\
    f(0,x)=f_0(x),\ x\in\mathbb{R}^d,
\end{array}\right.
\end{equation}
which are crucial to the proof of our main theorem later.

\begin{lemm}\label{heat}\cite{Danchin}
Let $s\in\mathbb{R}, 1\leq q,q_1,p,r\leq\infty$ with $q_1\leq q$. Assume $u_0$ in $\dot{B}^s_{p,r}$ or ${B}^s_{p,r}$, and $G$ in $\widetilde{L}^{q_1}_T(\dot{B}^s_{p,r})$ or $\widetilde{L}^{q_1}_T(^s_{p,r})$. Then \eqref{s1cuchong} has a unique solution $u$ in $\widetilde{L}^{q}_T(\dot{B}^{s+\frac{2}{q}}_{p,r})$ or $\widetilde{L}^{q}_T({B}^{s+\frac{2}{q}}_{p,r})$ and satisfies
$$ \|u\|_{\widetilde{L}^{q}_T(\dot{B}^{s+\frac{2}{q}}_{p,r})}\leq C_1\Big(\|u_0\|_{\dot{B}^s_{p,r}}+\|G\|_{\widetilde{L}^{q_1}_T(\dot{B}^{s+\frac{2}{q_1}-2}_{p,r})}\Big) $$
or
$$ \|u\|_{\widetilde{L}^{q}_T({B}^{s+\frac{2}{q}}_{p,r})}\leq C_1\Big(\|u_0\|_{{B}^s_{p,r}}+(1+T^{1+\frac{1}{q}-\frac{1}{q_1}})\|G\|_{\widetilde{L}^{q_1}_T({B}^{s+\frac{2}{q_1}-2}_{p,r})}\Big). $$
Moreover, if $\Delta_{-1} f=0$ in the periodic case, we have
$$ \|u\|_{\widetilde{L}^{q}_T({B}^{s+\frac{2}{q}}_{p,r})}\leq C_1\Big(\|u_0\|_{{B}^s_{p,r}}+\|G\|_{\widetilde{L}^{q_1}_T({B}^{s+\frac{2}{q_1}-2}_{p,r})}\Big). $$
\end{lemm}

\begin{lemm}\label{priori estimate}\cite{book}
Let $s\in [\max\{-\frac{d}{p},-\frac{d}{p'}\},\frac{d}{p}+1](s=1+\frac{d}{p},r=1; s=\max\{-\frac{d}{p},-\frac{d}{p'}\},r=\infty).$
There exists a constant $C$ such that for all solutions $f\in L^{\infty}([0,T];{B}^s_{p,r})$ of \eqref{s1} with initial data $f_0$ in ${B}^s_{p,r}$, and $g$ in $L^1([0,T];{B}^s_{p,r})$, we have, for a.e. $t\in[0,T]$,
\begin{align}
\|f(t)\|_{{B}^s_{p,r}}\leq& C\Big(\|f_0\|_{{B}^s_{p,r}}+\int_0^t V'(t')\|f(t')\|_{{B}^s_{p,r}}+\|g(t')\|_{{B}^s_{p,r}}{\ud}t'\Big)\notag\\
\leq&
e^{C_2 V(t)}\Big(\|f_0\|_{{B}^s_{p,r}}+\int_0^t e^{-C_2 V(t')}\|g(t')\|_{{B}^s_{p,r}}{\ud}t'\Big),
\end{align}
where $V(t)=\int_{0}^{t}\|\nabla v\|_{{B}^{\frac{d}{p}}_{p,r}\cap L^{\infty}}ds($if $s=1+\frac{1}{p},r=1$, $V'(t)=\int_{0}^{t}\|\nabla v\|_{{B}^{\frac{d}{p}}_{p,1}}ds).$
\end{lemm}

\begin{rema}\label{priori estimate1}\cite{book}
If ${\rm{div}}v=0$, we can get the same result with a better indicator: $\max\{-\frac{d}{p},-\frac{d}{p'}\}-1<s<\frac{d}{p}+1($or $s=\max\{-\frac{d}{p},-\frac{d}{p'}\}-1,r=\infty).$
\end{rema}

\begin{lemm}\label{priori estimate0}\cite{book}
Let ${\rm{div}} v=0$.
There exists a constant $C$ such that for all solutions $f\in L^{\infty}([0,T];{B}^0_{p,r})$ of \eqref{s1} with initial data $f_0$ in ${B}^0_{p,r}$, and $g$ in $L^1([0,T];{B}^s_{p,r})$, we have, for all $1\leq p,r\leq \infty$ and $t\in[0,T]$,
\begin{align}
\|f(t)\|_{{B}^s_{p,r}}\leq& C(1+\int_0^t V'(t')dt')(\|f_0\|_{{B}^s_{p,r}}+\int_0^t\|g(t')\|_{{B}^s_{p,r}}{\ud}t'),
\end{align}
where $V'(t)=\int_{0}^{t}\|\nabla v(t)\|_{L^{\infty}}ds.$
\end{lemm}

\begin{lemm}\cite{book}\label{Convergence}
Let $p\in[1,\infty) (p=\infty,~{\rm{div}} A^n=0)$. Define  $\overline{\mathbb{N}}=\mathbb{N}\cup \{\infty\}$. Suppose $f\in L^1([0,T];B^{\frac{d}{p}}_{p,1})$ and $a_0\in B^{\frac{d}{p}}_{p,1}$. For $n\in \overline{\mathbb{N}}$, denote by $a^n\in C([0,T];B^{\frac{d}{p}}_{p,1})$ the solution of
\begin{align}\label{an}
\left\{
\begin{array}{ll}
\partial_t a^n+A^n\cdot\nabla a^n=f,\\[1ex]
a^n|_{t=0}(x)=a_{0}(x).\\[1ex]
\end{array}
\right.
\end{align}
Assume that $\sup_{n\in\overline{\mathbb{N}} }\|A^n\|_{L^1_T(B^{\frac{d}{p}+1}_{p,1})}\leq C$. If $A^n\to A^\infty$ in $L^1_T(B^{\frac{d}{p}}_{p,1})$, then
the sequence $a^n\to a^\infty$ in $C([0,T];B^{\frac{d}{p}}_{p,1})$.
\end{lemm}

\begin{defi}\cite{book}
Let $a>0$, $\mu(r)$ be a continue non-zero and non-decreasing function from $[0,a]$ to $\mathbb{R}^+$, $\mu(0)=0$. We say that $\mu$ is an Osgood modulus of continuity if
$$\int_{0}^{a}\frac{1}{\mu(r)}{\ud}r=+\infty.$$
\end{defi}

\begin{lemm}\label{osgood}\cite{book}
Let $\rho$ be a measurable function from $[0,T]$ to $[0,a]$, $\gamma$ be a locally integrable function from $[0,T]$ to $\mathbb{R}^+$, and $\mu$ be an Osgood modulus of continuity. If for some $\rho_0\geq 0$,
$$\rho(t)\leq \rho_0+\int_{0}^{t}\gamma(s)\mu(\rho(s)){\ud}s\quad for\quad a.e.\quad t\in[0,T],$$
then we have
\begin{equation}
-M(\rho(t))+M(\rho_0)\leq \int_{0}^{t}\gamma(s){\ud}s\quad with\quad M(x)=\int_{x}^{a}\frac{{\ud}r}{\mu(r)}.
\end{equation}
If $\mu(r)=r$, we obtain the Gronwall inequality:
\begin{align}
\rho(t)\leq \rho_0e^{\int_{0}^{t}\gamma(s){\ud}s}.
\end{align}
If $\mu(r)=cr(1+\ln(e+r))$, it's easy to check that it is still an Osgood modulus of continuity and have
\begin{align}\label{osgood1}
\rho(t)\leq C\rho_0e^{cte^{\int_{0}^{t}c\gamma(s){\ud}s}}.
\end{align}
If $\mu(r)=r+r\ln(e+c/r)$, similarly, one has
\begin{align}\label{osgood2}
\rho(t)\leq C\rho_0\frac{e^{e^{\int_{0}^{t}\gamma(s){\ud}s}}}{c-\rho_0(e^{\int_{0}^{t}\gamma(s){\ud}s}-e)}\leq C\rho_0e^{e^{\int_{0}^{t}\gamma(s){\ud}s}}.
\end{align}
where the second inequality holds for sufficient small $\rho_0$. %small enough such that $\rho_0\leq\frac{c}{2(e^{\int_{0}^{t}\gamma(s){\ud}s}-e)}$, then
\end{lemm}
\par

\section{Local existence and uniqueness}

Without loss of generality, we prove the local well-posedness in $\mathbb{R}^d$, since the period case $\mathbb{S}^d$ is similar. For convenience, we denote $C_{E_0}\approx C(1+E_0+e^{E_0})$ for $C$ large enough in the proof of Theorem \ref{theorem}. We divide the proof of local existence and uniqueness into four steps:

\textbf{Step 1: An iterative scheme.}

Define the first term $(u^0,b^0):=(e^{t\Delta}u_0,e^{t\Delta}b_0)$. Then we introduce a sequence $(u^n,b^n)$ with the initial data $(u^n_0,b^n_0)$ by solving the following linear transport equation and heat conductive equation:
\begin{equation}\label{sp1}
\left\{\begin{array}{lll}
u^{n+1}_t+u^{n}\nabla u^{n+1}=b^{n}\nabla b^{n}+\frac{\nabla {\rm{div}}}{-\Delta}(b^n\nabla b^n-u^n\nabla u^n),\\
b^{n+1}_t-\Delta b^{n+1}=-u^{n}\nabla b^{n}+b^{n}\nabla u^{n} ,\\
(u^n_0,b^n_0):=({S}_nu_0,{S}_nb_0),
\end{array}\right.
\end{equation}
where ${S}_ng:=\sum_{k<n}{\Delta}_k g$ which makes sense in nonhomogeneous Besov spaces.

\textbf{Step 2: Uniform estimates.}

Taking advantage of Lemmas \ref{heat}--\ref{priori estimate}, we can obtain the uniform boundness of the approximate solution sequences $(u^n,b^n)\in E^p_T$. Now we claim that there exists a $T$ independent of $n$ such that the approximate solutions $(u^n,b^n)$ satisfy the following estimations :
$$(H_1):\quad \|b^{n}\|_{L^{\infty}_T ({B}^{\frac{d}{p}}_{p,1})}+\|u^{n}\|_{L^{\infty}_T ({B}^{\frac{d}{p}-1}_{p,1})}\leq 6E_0, $$
$$\quad\quad\quad~~~(H_2):\quad \|b^{n}\|_{A_T}\leq 2a,\quad A_T:={L^{2}_T({B}^{\frac{d}{p}+1}_{p,1})\cap L^{1}_T({B}^{\frac{d}{p}+2}_{p,1}}),$$
where $E_0:=\|b_0\|_{{B}^{\frac{d}{p}}_{p,1}}+\|u_0\|_{{B}^{\frac{d}{p}+1}_{p,1}}$. Now we suppose that $a$ is small enough such that %($a$ will be determined later):
\begin{align}\label{xxiugailsp2}
a\approx \frac{1}{24C},
\end{align}
\begin{align}\label{lsp2}
T\leq  \min\{1,\frac{1}{(96CE_0)^2},\frac{1}{(96Ca)^2},\frac{1}{72CE_0},\frac{\ln2}{12CE_0}\} \quad\text{and} \quad \|e^{t\Delta}u_0\|_{A_T}\leq a,
\end{align}
where $C=\max\{C_1,C_2\}$ and $C_1,C_2$  are the constants in Lemma \ref{heat}--Lemma\ref{priori estimate} (Indeed, we can take $C$ more large as we need).

It's easy to check that $(H_1)-(H_2)$ hold true for $n=0$. Now we will show that if $(H_1)-(H_2)$ hold true for $n$, then they hold true for $n+1$. In fact, \eqref{xxiugailsp2}--\eqref{lsp2} and Lemma \ref{heat}--Lemma\ref{priori estimate} together yield
\begin{align*}
\|b^{n+1}\|_{A_T}\leq& \|e^{t\Delta}b_0\|_{A_T}+C_1(1+T^{\frac{1}{p}})\|-u^{n}\nabla b^{n}+b^{n}\nabla u^{n}\|_{L^{1}_T({B}^{\frac{d}{p}}_{p,1})}\\
\leq& a+12CT^{\frac{1}{2}}E_0a\leq 2a,
\end{align*}
\begin{align*}
\|u^{n+1}\|_{L^{\infty}_T{B}^{\frac{d}{p}+1}_{p,1}}\leq& e^{6TE_0C}(\|u_0\|_{{B}^{\frac{d}{p}+1}_{p,1}}+\|b^{n}\nabla b^{n}+\frac{\nabla div}{-\Delta}(b^n\nabla b^n-u^n\nabla u^n)\|_{L^{1}_T{B}^{\frac{d}{p}-1}_{p,1}})
\\
\leq& e^{6TE_0C}E_0+C(\|b^{n}\|_{{B}^{\frac{d}{p}}_{p,1}}\|b^{n}\|_{L^1{B}^{\frac{d}{p}+2}_{p,1}}+T\|u^{n}\|^2_{L^{\infty}{B}^{\frac{d}{p}}_{p,1}})\\
\leq& 2E_0+12CE_0a+36CTE^2_0\leq 3E_0
\end{align*}
and
\begin{align*}
\|b^{n+1}\|_{L^{\infty}_T({B}^{\frac{d}{p}}_{p,1})}\leq& \|e^{t\Delta}b_0\|_{{B}^{\frac{d}{p}}_{p,1}}+\|-u^{n}\nabla b^{n}+b^{n}\nabla u^{n}\|_{L^{1}_T({B}^{\frac{d}{p}}_{p,1})}
\\
\leq& \|e^{t\Delta}b_0\|_{{B}^{\frac{d}{p}}_{p,1}}+CT^{\frac{1}{2}}\|u^{n}\|_{L^{\infty}_T({B}^{\frac{d}{p}+1}_{p,1})}\|b^{n}\|_{L^{2}_T({B}^{\frac{d}{p}+1}_{p,1})}\\
\leq&E_0+12CT^{\frac{1}{2}}E_0a=3E_0,
\end{align*}
which implies $(H_1)-(H_2)$ hold true for $n+1$.

Then, we need to obtain the relationship between the existence time $T$ and the initial data by \eqref{lsp2}. It is easy to deduce that
$$T\leq T_0:=\min\{1,\frac{1}{(96CE_0)^2},\frac{1}{(96Ca)^2},\frac{1}{72CE_0},\frac{\ln2}{12CE_0}\}.$$
Now we turn to study the condition $\|e^{t\Delta}u_0\|_{A_T}\leq a$ of (\ref{lsp2}). For this purpose, we have to classify the initial data.\\
(1) For $\|u_0\|_{{B}^{\frac{d}{p}-1}_{p,1}}\leq a$, we have
$$\|e^{t\Delta}u_0\|_{A_T}\leq \|u_0\|_{{B}^{\frac{d}{p}-1}_{p,1}}\leq a. $$
(2) For $\|u_0\|_{{B}^{\frac{d}{p}-1}_{p,1}}> a$, since $b_0\in{B}^{\frac{d}{p}}_{p,1}$, there exists an integer $j_0$ such that ($j_0$ may not be unique):
\begin{align}\label{lsp5.5}
\sum_{|j|\geq j_0}\|{\Delta}_ju_0\|_{L^p}2^{\frac{d}{p}j}< \frac{a}{4}.
\end{align}
Defining  $T_1:=\frac{a}{4}\frac{1}{2^{2j_0}\|u_0\|_{{B}^{\frac{d}{p}-1}_{p,1}}}$ and $T_2:=\frac{a^2}{4^2}\frac{1}{2^{2j_0}\|u_0\|^2_{{B}^{\frac{d}{p}-1}_{p,1}}}$, we get
\begin{align}\label{lsp6}
&\|e^{t\Delta}u_0\|_{L^{1}_{T_1}({B}^{\frac{d}{p}+2}_{p,1})}                         \notag\\
&\leq \sum_{|j|\leq j_0}\int_{0}^{T_1}\|e^{t\Delta}\dot{\Delta}_ju_0\|_{L^p}2^{(\frac{d}{p}+2)j}{\ud}t+\sum_{|j|> j_0}\int_{0}^{T_1}e^{-t2^{2j}}\|\dot{\Delta}_ju_0\|_{L^p}2^{(\frac{d}{p}+2)j}{\ud}t         \notag\\
&\leq 2^{2j_0}\sum_{|j|\leq j_0}\int_{0}^{T_1}\|\dot{\Delta}_ju_0\|_{L^p}2^{(\frac{d}{p}1)j}{\ud}t+\sum_{|j|> j_0}\int_{0}^{T_1}e^{-t2^{2j}}\|\dot{\Delta}_ju_0\|_{L^p}2^{(\frac{d}{p}+2)j}dt         \notag\\
&\leq 2^{2j_0}T_1\|u_0\|_{{B}^{\frac{d}{p}}_{p,1}}+\sum_{|j|> j_0}(1-e^{-T_22^{2j}})\|\dot{\Delta}_ju_0\|_{L^p}2^{(\frac{d}{p})j}                    \notag\\
&\leq 2^{2j_0}T_1\|u_0\|_{{B}^{\frac{d}{p}}_{p,1}}+\sum_{|j|> j_0}\|\dot{\Delta}_ju_0\|_{L^p}2^{(\frac{d}{p})j}\leq \frac{a}{2}
\end{align}
and
\begin{align}\label{lsp7}
&\|e^{t\Delta}u_0\|_{L^{2}_{T_2}({B}^{\frac{d}{p}+1}_{p,1})}\notag \\
&\leq \sum_{|j|\leq j_0}[\int_{0}^{T_2}\|e^{t\Delta}\dot{\Delta}_ju_0\|^2_{L^p}{\ud}t]^{\frac{1}{2}}2^{(\frac{d}{p}+1)j}+\sum_{|j|> j_0}[\int_{0}^{T_2}(e^{-t2^{2j}}\|\dot{\Delta}_ju_0\|_{L^p})^2{\ud}t]^{\frac{1}{2}} 2^{(\frac{d}{p}+1)j}\notag\\
&\leq 2^{j_0}T^{\frac{1}{2}}_2\|u_0\|_{{B}^{\frac{d}{p}}}+\sum_{|j|> j_0}(1-e^{-T_22^{2j}})^{\frac{1}{2}}\|\dot{\Delta}_ju_0\|_{L^p}2^{\frac{d}{p}j}        \notag \\
&\leq 2^{j_0}T^{\frac{1}{2}}_2\|u_0\|_{{B}^{\frac{d}{p}}}+\sum_{|j|> j_0}\|\dot{\Delta}_ju_0\|_{L^p}2^{\frac{d}{p}j}\leq \frac{a}{2}.
\end{align}
Letting $T=\min\{T_0,T_1,T_2\}$, we see
$$\|e^{t\Delta}u_0\|_{A_T}\leq a.$$

Finally, if we choose $T$ satisfying
\begin{equation}\label{lsp8}
  T=\begin{cases}
    T_0, & \|u_0\|_{{B}^{\frac{d}{p}-1}_{p,1}}\leq a,   \\
    \min\{T_0, T_1, T_2\}, & \|u_0\|_{{B}^{\frac{d}{p}-1}_{p,1}}>a,
  \end{cases}
\end{equation}
then \eqref{lsp2} holds. For this $T$, we find the approximate sequence $(u^n,b^n)$ is uniformly bounded in $E^p_T$.

\begin{rema}\label{j0}
From \eqref{lsp8}, we know that if the initial data is small, the local existence time $T$ depends only on $E_0$. However, for large initial data, the local existence time $T$ depends on both $E_0$ and the $j_0$ which satisfies \eqref{lsp5.5}.
\end{rema}

\textbf{Step 3: Existence of a solution.}

In this step, we will adopt the compactness argument in Besov spaces for the approximate sequence $(u^n,b^n)$ to get some solution $(u,b)$ of \eqref{sp00}, which is similar to the process in \cite{book,Li1,miao}. In fact, it is a routine route to verify that $(u,b)$ satisfies \eqref{sp00}, and here we omit it. Then, following the similar arguement of Theorem 3.19 in \cite{book}, we can prove $(u,b)\in E^p_T$.

\textbf{Step 4: Uniqueness.}

\begin{prop}\label{uniqueness}
Let $d\geq 2,~p\in[1,\infty]$. Suppose that $(u^1,b^1),(u^2,b^2)\in E^p_T$ are two corresponding solutions to \eqref{sp00} given by Step 1 with the initial data $(u^1_{0},b^1_{0})$ and $(u^2_{0},b^2_{0})$ respectively. Denote $\delta u=u^1-u^2$ and $\delta b=b^1-b^2$. Then for any $t\in[0,T]$, we have
\begin{align}\label{uniquegj1}
\|\delta b\|_{L^{\infty}_T ({B}^{\frac{d}{p}}_{p,1})\cap L^{1}_T ({B}^{\frac{d}{p}+2}_{p,1})}+\|\delta u\|_{L^{\infty}_T ({B}^{\frac{d}{p}+1}_{p,1})}\leq
e^{A(T)}(\|\delta b_0\|_{{B}^{\frac{d}{p}}_{p,1}}+\|\delta u_0\|_{{B}^{\frac{d}{p}+1}_{p,1}}+\int_{0}^{T}\|u^{2}\|_{{B}^{\frac{d}{p}+2}_{p,1}}\|\delta u\|_{{B}^{\frac{d}{p}+1}_{p,1}}{\ud}\tau).
\end{align}
Moreover, if $\|\delta b_0\|_{{B}^{\frac{d}{p}-1}_{p,1}}+\|\delta u_0\|_{{B}^{\frac{d}{p}}_{p,1}}$ is small enough, we have
\begin{align}\label{uniquegj2}
\|\delta b\|_{L^{\infty}_T ({B}^{\frac{d}{p}-1}_{p,\infty})\cap L^{1}_T ({B}^{\frac{d}{p}+1}_{p,\infty})}+\|\delta u\|_{L^{\infty}_T ({B}^{\frac{d}{p}}_{p,\infty})}\leq
e^{e^{A(T)}}(\|\delta b_0\|_{{B}^{\frac{d}{p}-1}_{p,1}}+\|\delta u_0\|_{{B}^{\frac{d}{p}}_{p,1}}),
\end{align}
where $A(T):=C_{\bar{E}_0}\int_{0}^{T}\|u^1\|_{{B}^{\frac{d}{p}+1}_{p,1}}+\|u^2\|_{{B}^{\frac{d}{p}+1}_{p,1}}+\|b^1\|_{{B}^{\frac{d}{p}+2}_{p,1}}+\|b^1\|_{{B}^{\frac{d}{p}+2}_{p,1}}{\ud}\tau$ and $\bar{E}_0:=\|b^1_0,b^2_0\|_{{B}^{\frac{d}{p}}_{p,1}}+\|u^1_0,u^2_0\|_{{B}^{\frac{d}{p}+1}_{p,1}}.$
\end{prop}
\begin{proof}
%Without loss of generality, one can set $E^1_0\neq 0$, otherwise we deduce that $u_1=b^1=0$, a trivial case.
First, since $(u_1,b_1),(u_2,b_2)\in E^p_T$ are two corresponding solution to \eqref{sp00} given by Step 1, we can let $T$ be their common lifespan. Then, $(\delta u,\delta b)$ solves
\begin{equation}\label{unique1}
\left\{\begin{array}{lll}
(\delta u)_t+u^1\nabla(\delta u)+(\delta u)\nabla u^{2}=b^1\nabla (\delta b)+(\delta b)\nabla b^{2}-(\nabla P_1-\nabla P_2),\\
(\delta b)_t-\Delta(\delta b)+u^1\nabla(\delta b)+(\delta u)\nabla b^{2}=b^1\nabla(\delta u)+(\delta b)\nabla u^{2}  ,\\
(\delta u,\delta b)|_{t=0}=(u^1_0-u^2_0,b^1_0-b^2_0).
\end{array}\right.
\end{equation}
where $\nabla P_1-\nabla P_2=\frac{\nabla {\rm{div}}}{-\Delta}\big(b^1\nabla (\delta b) +(\delta b)\nabla b^2-u^1\nabla (\delta u) -(\delta u)\nabla u^2\big)$.
Lemma\ref{heat}--Lemma \ref{priori estimate} and \eqref{lsp2} give that
\begin{align}\label{unique2}
\|\delta u\|_{{B}^{\frac{d}{p}+1}_{p,1}}&\leq \|u^1_0-u^2_0\|_{{B}^{\frac{d}{p}+1}_{p,1}}\notag\\
&\quad +\int_{0}^{t}(\|u^1\|_{{B}^{\frac{d}{p}+1}_{p,1}}+\|u^{2}\|_{{B}^{\frac{d}{p}+2}_{p,1}})\|\delta u\|_{{B}^{\frac{d}{p}+1}_{p,1}}+\|b^1,b^2\|_{{B}^{\frac{d}{p}}_{p,1}}\|\delta b\|_{{B}^{\frac{d}{p}+2}_{p,1}}+\|b^1,b^2\|_{{B}^{\frac{d}{p}+2}_{p,1}}\|\delta b\|_{{B}^{\frac{d}{p}}_{p,1}}{\ud}\tau\notag\\
&\leq \|u^1_0-u^2_0\|_{{B}^{\frac{d}{p}-1}_{p,1}}+12E_0C\|\delta b\|_{L^1_T{B}^{\frac{d}{p}+2}_{p,1}}\notag\\
&\quad+\int_{0}^{t}(\|b^1,b^2\|_{{B}^{\frac{d}{p}+2}_{p,1}}+C_{E_0})(\|\delta u\|_{{B}^{\frac{d}{p}+1}_{p,1}}+\|\delta b\|_{{B}^{\frac{d}{p}}_{p,1}})+\|u^{2}\|_{{B}^{\frac{d}{p}+2}_{p,1}}\|\delta^nu\|_{{B}^{\frac{d}{p}+1}_{p,1}}{\ud}\tau
\end{align}
and
\begin{align}\label{unique3}
\|\delta b\|_{L^{\infty}_{T}{B}^{\frac{d}{p}}_{p,1}\cap L^2_T({B}^{\frac{d}{p}+1}_{p,1})\cap L^1_T({B}^{\frac{d}{p}+2}_{p,1})}
\leq& \|b^{1}_0-b^{2}_0\|_{{B}^{\frac{d}{p}}_{p,1}}\notag\\
&+\int_{0}^{t}\|b^1,b^{2}\|_{{B}^{\frac{d}{p}+1}_{p,1}}\|\delta u\|_{{B}^{\frac{d}{p}+1}_{p,1}}
+\|u^1,u^{2}\|_{{B}^{\frac{d}{p}+1}_{p,1}}\|\delta b\|_{{B}^{\frac{d}{p}+1}_{p,1}}{\ud}\tau\notag\\
\leq& \|b^{1}_0-b^{2}_0\|_{{B}^{\frac{d}{p}}_{p,1}}\notag\\
&+\int_{0}^{t}\|b^1,b^{2}\|_{{B}^{\frac{d}{p}+1}_{p,1}}\|\delta u\|_{{B}^{\frac{d}{p}+1}_{p,1}}{\ud}\tau
+ 12E_0T^{\frac{1}{2}}\|\delta b\|_{L^2_{T}({B}^{\frac{d}{p}+1}_{p,1})}\notag\\
\leq& \|b^{1}_0-b^{2}_0\|_{{B}^{\frac{d}{p}}_{p,1}}\notag\\
&+\int_{0}^{t}\|b^1,b^{2}\|_{{B}^{\frac{d}{p}+1}_{p,1}}\|\delta u\|_{{B}^{\frac{d}{p}+1}_{p,1}}{\ud}\tau
+ \frac{1}{4}\|\delta b\|_{L^2_{T}({B}^{\frac{d}{p}+1}_{p,1})}
\end{align}
Combining \eqref{unique2} and \eqref{unique3}$\times (24E_0+1)C$, we have by Gronwall's inequality
\begin{align*}
&\|\delta u\|_{L^{\infty}({B}^{\frac{d}{p}-1}_{p,1})}+\|\delta b\|_{L^{\infty}({B}^{\frac{d}{p}}_{p,1})\cap L^2_T({B}^{\frac{d}{p}+1}_{p,1})\cap L^1_T({B}^{\frac{d}{p}+2}_{p,1})}\\
&\leq e^{A(T)}\Big(\|b^{1}_0-b^{2}_0\|_{{B}^{\frac{d}{p}}_{p,1}}+\|u^1_0-u^2_0\|_{{B}^{\frac{d}{p}-1}_{p,1}}+\int_{0}^{T}\|u^{2}\|_{{B}^{\frac{d}{p}+2}_{p,1}}\|\delta^nu\|_{{B}^{\frac{d}{p}+1}_{p,1}}{\ud}\tau\Big).
\end{align*}
That is \eqref{uniquegj1}.

Then, similarly, according to Lemma \ref{heat}--\ref{priori estimate} and \eqref{lsp2}, we get
\begin{align}
\|\delta u\|_{{B}^{\frac{d}{p}}_{p,\infty}}\leq& \|u^1_0-u^2_0\|_{{B}^{\frac{d}{p}}_{p,\infty}}\notag\\
&+C\int_{0}^{t}\big(\|u^1,u^{2}\|_{{B}^{\frac{d}{p}+1}_{p,1}}\|\delta^nu\|_{{B}^{\frac{d}{p}}_{p,1}}+\|b^2\|_{{B}^{\frac{d}{p}+2}_{p,1}}\|\delta b\|_{{B}^{\frac{d}{p}-1}_{p,\infty}}+\|b^1,b^2\|_{{B}^{\frac{d}{p}}_{p,1}}\|\delta b\|_{{B}^{\frac{d}{p}+1}_{p,\infty}}\big){\ud}\tau\notag\\
\leq& \|u^1_0-u^2_0\|_{{B}^{\frac{d}{p}-1}_{p,1}}\notag\\
&+\int_{0}^{t}12E_0C\big(\|\delta u\|_{{B}^{\frac{d}{p}}_{p,1}}+\|b^2\|_{{B}^{\frac{d}{p}+2}_{p,1}}\|\delta b\|_{{B}^{\frac{d}{p}-1}_{p,\infty}}\big){\ud}\tau
+12E_0C\|\delta b\|_{L^1_T{B}^{\frac{d}{p}+1}_{p,\infty}}\label{unique4}
\end{align}
and
\begin{align}\label{unique5}
\|\delta b\|_{L^{\infty}_{T}({B}^{\frac{d}{p}}_{p,1})\cap L^2_T{B}^{\frac{d}{p}+1}_{p,1}\cap L^1_T{B}^{\frac{d}{p}+2}_{p,1}}
&\leq \|b^{1}_0-b^{2}_0\|_{{B}^{\frac{d}{p}}_{p,1}}
+\int_{0}^{t}\|b^1,b^{2}\|_{{B}^{\frac{d}{p}}_{p,1}}\|\delta u\|_{{B}^{\frac{d}{p}}_{p,1}}
+\|u^1,u^{2}\|_{{B}^{\frac{d}{p}+1}_{p,1}}\|\delta b\|_{{B}^{\frac{d}{p}}_{p,\infty}}{\ud}\tau\notag\\
&\leq \|b^{1}_0-b^{2}_0\|_{{B}^{\frac{d}{p}}_{p,\infty}}+\int_{0}^{t}\|b^1,b^{2}\|_{{B}^{\frac{d}{p}+1}_{p,1}}\|\delta u\|_{{B}^{\frac{d}{p}+1}_{p,1}}{\ud}\tau
+ 12E_0T^{\frac{1}{2}}\|\delta b\|_{L^2_{T}({B}^{\frac{d}{p}+1}_{p,\infty})}\notag\\
&\leq \|b^{1}_0-b^{2}_0\|_{{B}^{\frac{d}{p}}_{p,\infty}}+\int_{0}^{t}12E_0C\|\delta u\|_{{B}^{\frac{d}{p}+1}_{p,1}}{\ud}\tau
+ \frac{1}{4}\|\delta b\|_{L^2_{T}({B}^{\frac{d}{p}+1}_{p,\infty})}.
\end{align}
By use of \eqref{unique4} and \eqref{unique5}$\times 24E_0C+1$, we see from Gronwall's inequality that
\begin{align}\label{unique6}
&\|\delta u\|_{{B}^{\frac{d}{p}}_{p,\infty}}+\|\delta b\|_{{B}^{\frac{d}{p}-1}_{p,\infty}\cap L^2_T{B}^{\frac{d}{p}}_{p,\infty}\cap L^1_T{B}^{\frac{d}{p}+1}_{p,\infty}}\notag\\
&\leq C_{E_0}(\|b^{1}_0-b^{2}_0\|_{{B}^{\frac{d}{p}}_{p,1}}+\|u^1_0-u^2_0\|_{{B}^{\frac{d}{p}-1}_{p,1}}
+\int_{0}^{t}\|\delta u\|_{{B}^{\frac{d}{p}+1}_{p,1}}{\ud}\tau).
\end{align}
Taking advantage of the interpolation inequality, we have
\begin{align}\label{interpolation}
\|\delta u\|_{{B}^{\frac{d}{p}}_{p,1}}
\leq C\|\delta u\|_{{B}^{\frac{d}{p}}_{p,\infty}}\ln(e+\frac{\|\delta u\|_{{B}^{\frac{d}{p}+1}_{p,1}}}{\|\delta u\|_{{B}^{\frac{d}{p}}_{p,\infty}}}) ,
\end{align}
which together with \eqref{unique6} yields that
\begin{align}\label{sppp2}
&\|\delta u\|_{L^{\infty}_{T}({B}^{\frac{d}{p}}_{p,\infty})}+\|\delta b\|_{L^{\infty}_{T} ({B}^{\frac{d}{p}-1}_{p,\infty})\cap L^2_T ({B}^{\frac{d}{p}}_{p,\infty})\cap L^1_T ({B}^{\frac{d}{p}+1}_{p,\infty})}\notag\\
&\leq C_{E_0}(\|b^{1}_0-b^{2}_0\|_{ {B}^{\frac{d}{p}}_{p,1}}+\|u^1_0-u^2_0\|_{ {B}^{\frac{d}{p}-1}_{p,1}}
+ \int_{0}^{t}\|\delta u\|_{{B}^{\frac{d}{p}}_{p,\infty}}
\ln(e+\frac{C_{E_0}}{\|\delta u\|_{ {B}^{\frac{d}{p}}_{p,\infty}}}){\ud}\tau %\notag\\
%&\leq C_{E_0}(\|b^{1}_0-b^{2}_0\|_{ {B}^{\frac{d}{p}}_{p,1}}+\|u^1_0-u^2_0\|_{ {B}^{\frac{d}{p}-1}_{p,1}}
%+ \int_{0}^{t}(\|\delta u\|_{ {B}^{\frac{d}{p}}_{p,\infty}}+\|\delta b\|_{ {B}^{\frac{d}{p}-1}_{p,\infty}})
%ln(e+\frac{C_{E_0}}{\|\delta_ju\|_{ {B}^{\frac{d}{p}}_{p,\infty}}+\|\delta b\|_{{B}^{\frac{d}{p}-1}_{p,\infty}}}){\ud}\tau).
\end{align}
Letting $\mu(r)=r\ln(e+\frac{C_{E_0}}{r}),\gamma(s)=C_{E_0}$ in Lemma \ref{osgood}, we find
\begin{align*}
\|\delta u\|_{L^{\infty}_t({B}^{\frac{d}{p}}_{p,\infty})}+\|\delta b\|_{ {B}^{\frac{d}{p}-1}_{p,\infty}\cap L^2_T ({B}^{\frac{d}{p}}_{p,\infty})\cap L^1_T({B}^{\frac{d}{p}+1}_{p,\infty})}
\leq e^{e^{A(T)}}(\|b^{1}_0-b^{2}_0\|_{ {B}^{\frac{d}{p}}_{p,1}}+\|u^1_0-u^2_0\|_{ {B}^{\frac{d}{p}-1}_{p,1}}).
\end{align*}
That is \eqref{uniquegj2}.
\end{proof}

Appealing to \eqref{uniquegj2} in Proposition \ref{uniqueness} and setting $\|\delta b_0\|_{{B}^{\frac{d}{p}-1}_{p,1}}+\|\delta u_0\|_{{B}^{\frac{d}{p}}_{p,1}}=0$, one can easily obtain the uniqueness of \eqref{sp00}.

\section{Continuous dependence}

Before proving the continuous dependence of solutions to \eqref{sp00}, we need to prove that if $(u^n_0,b^n_0)$ tends to $(u_0,b_0)$ in ${B}^{\frac{d}{p}+1}_{p,1}\times {B}^{\frac{d}{p}}_{p,1}$, then there exists a lifespan $T^n$ corresponding to $(u^n_0,b^n_0)$ such that $T^n\rightarrow T$  where $T$ is a lifespan corresponding to the initial data $u_0$ in \eqref{lsp8}. This implies that $T-\epsilon$ (for some small $\epsilon$) is a common lifespan both for $(u^n,b^n)$ and $(u,b)$ when $n$ is sufficiently large. We first give a useful lemma:
\begin{lemm}\label{gj}
Let $(u_0,b_0)\in {B}^{\frac{d}{p}-1}_{p,1}\times{B}^{\frac{d}{p}}_{p,1}$ be the initial data of (\ref{sp00}) with $p\leq 2d$. If there exists another initial data $(u^n_0,b^n_0)\in {B}^{\frac{d}{p}+1}_{p,1}\times{B}^{\frac{d}{p}}_{p,1}$ such that $\|u^n_0-u_0\|_{{B}^{\frac{d}{p}+1}_{p,1}}+\|b^n_0-b_0\|_{{B}^{\frac{d}{p}}_{p,1}}\rightarrow 0\ (n\rightarrow\infty)$, then we can construct a lifespan $T^n$ corresponding to $(u^n_0,b^n_0)$ such that
$$T^n\rightarrow T,\quad\quad n\rightarrow\infty ,$$
where the lifespan $T$ corresponds to $(u_0,b_0)$.
\end{lemm}
\begin{proof}
The proof is similar to Lemma 4.1 in \cite{weikui}, and we omit it here.
\end{proof}

\begin{rema}
From Lemma \ref{gj}, letting $T$ be the lifespan time of $(u^{\infty},b^{\infty})$, then we can define a $T^n$ corresponding to $(u^n,b^n)$ such that $T^n\rightarrow T,n\rightarrow\infty$. That is, for any  fixed small $\epsilon >0$, there exists an integer $N$, when $n\geq N$, we have
$$|T^n-T|< \epsilon.$$
Thus, we can choose $T-\epsilon$ as the common lifespan both for $(u^{\infty},b^{\infty})$ and $(u^{n},b^{n})$, which is independent of $n$.
\end{rema}
Now we begin to prove the continuous dependence.
\begin{theo}\label{continuous dependence}
Let $1\leq p\leq \infty$. Assume that $(u^n,b^n)_{n\in\overline{\mathbb{N}}}$ be the solution to \eqref{sp00} with the initial data $(u^n_0,b^n_0)_{n\in\overline{\mathbb{N}}}$. If $(u^n_0,b^n_0)_{n\in\mathbb{N}}$ tends to $(u^{\infty}_0,b^{\infty}_0)$ in ${B}^{\frac{d}{p}+1}_{p,1}\times{B}^{\frac{d}{p}}_{p,1}$, then there exists a positive ${T}$ independent of $n$ such that $(u^n,b^n)_{n\in{\mathbb{N}}}$ tends to $(u^{\infty},b^{\infty})$ in $C_{T}({B}^{\frac{d}{p}+1}_{p,1})\times \Big(C_{T}({B}^{\frac{d}{p}}_{p,1})\cap L^{1}_{T}({B}^{\frac{d}{p}+2}_{p,1})\Big)$.
\end{theo}
\begin{proof}
Our aim is to estimate $\|u^n-u^{\infty}\|_{L^{\infty}_{T}({B}^{\frac{d}{p}+1}_{p,1})}$ and $\|b^n-b^{\infty}\|_{L^{\infty}_{T}({B}^{\frac{d}{p}}_{p,1})\cap L^{1}_{T}({B}^{\frac{d}{p}+2}_{p,1})}$ when $n\rightarrow\infty$. Note that
\begin{equation}\label{frequence}
  \left\{\begin{array}{l}
    \|u^n-u^{\infty}\|_{{L^{\infty}_{T}({B}^{\frac{d}{p}+1}_{p,1})}}\\
    \leq\|u^n-u^n_j\|_{{L^{\infty}_{{T}}({B}^{\frac{d}{p}+1}_{p,1})}}+\|u^n_j-u^{\infty}_j\|_{{L^{\infty}_{{T}}({B}^{\frac{d}{p}+1}_{p,1})}}
    +\|u^{\infty}_j-u^{\infty}\|_{{L^{\infty}_{{T}}({B}^{\frac{d}{p}+1}_{p,1})}},\\
    \|b^n-b^{\infty}\|_{L^{\infty}_{T}({B}^{\frac{d}{p}}_{p,1})\cap L^{1}_{{T}}({B}^{\frac{d}{p}+2}_{p,1})}\leq\|b^n-b^n_j\|_{L^{\infty}_{T}({B}^{\frac{d}{p}}_{p,1})\cap L^{1}_{{T}}({B}^{\frac{d}{p}+2}_{p,1})}\\
    \qquad\qquad\qquad\qquad\qquad\quad+\|b^n_j-b^{\infty}_j\|_{L^{\infty}_{T}({B}^{\frac{d}{p}}_{p,1})\cap L^{1}_{{T}}({B}^{\frac{d}{p}+2}_{p,1})}
    +\|b^{\infty}_j-b^{\infty}\|_{L^{\infty}_{T}({B}^{\frac{d}{p}}_{p,1})\cap L^{1}_{{T}}({B}^{\frac{d}{p}+2}_{p,1})},
  \end{array}\right.
\end{equation}
where
\begin{align*}
&(u^n,b^n) \text{ corresponds to the initial data } (u^n_0,b^n_0), \quad n\in \overline{\mathbb{N}},\\
&(u^{n}_j,b^{n}_j) \text{ corresponds to the initial data } (\dot{S}_ju^n_0,\dot{S}_jb^n_0),  \quad n, j\in \overline{\mathbb{N}}.
\end{align*}
Using Lemma \ref{gj}, we find that $T-\epsilon$ (we still write it as $T$) is the common lifespan for $(u^{n},b^{n})$, $(u^{n}_j,b^{n}_j)$, $(u^{\infty},b^{\infty})$ and $(u^{\infty}_j,b^{\infty}_j)$ when $n,j$ are large enough. Since $(u^n_0,b^n_0)\rightarrow(u^{\infty}_0,b^{\infty}_0)$ and $({S}_ju^n_0,{S}_jb^n_0)\rightarrow(u^{n}_0,b^{n}_0)$ in ${B}^{\frac{d}{p}+1}_{p,1}\times{B}^{\frac{d}{p}}_{p,1}$, taking the similar argument as in Step 2, we see that for any large $n,j$,
\begin{align}\label{jie}
\|u^n,u^n_j\|_{L^{\infty}_{T}({B}^{\frac{d}{p}+1}_{p,1})},\quad\|b^n,b^n_j\|_{L^{\infty}_{T}({B}^{\frac{d}{p}}_{p,1})}\leq C_{E_0},\quad\|b^n,b^n_j\|_{L^{2}_{T}({B}^{\frac{d}{p}+1}_{p,1})\cap L^{1}_{T}({B}^{\frac{d}{p}+2}_{p,1})}\leq 2a,
\end{align}
where $E^n_0:=\|u^{n}_0\|_{{B}^{\frac{d}{p}-1}_{p,1}}+\|b^{n}_0\|_{{B}^{\frac{d}{p}}_{p,1}}$, $a$ is a small quantity satisfying \eqref{xxiugailsp2} and $T$ satisfies \eqref{lsp2}.
 %Similarly, as $(\dot{S}_ju^n_0,\dot{S}_jb^n_0)$ tends to $(u^{n}_0,b^{n}_0)$ in ${B}^{\frac{d}{p}-1}_{p,1}\times{B}^{\frac{d}{p}}_{p,1}$, for any large $n$ and $j$, we have
%$$\|u^n_j\|_{L^{\infty}_{T}{B}^{\frac{d}{p}-1}_{p,1}}+\|b^n_j\|_{L^{\infty}_{T}{B}^{\frac{d}{p}}_{p,1}}\leq C_{E^n_0}\leq C_{E_0},\quad\|u^n_j\|_{L^{2}_{T}{B}^{\frac{d}{p}}_{p,1}\cap L^{1}_{T}{B}^{\frac{d}{p}+1}_{p,1}}\leq 2a\leq \frac{1}{4C_1}.$$
For any $t\in [0,T]$, we now estimate \eqref{frequence} in three parts.

\textbf{1. Estimate }$\|u^n_j-u^{\infty}_j\|_{{L^{\infty}_{T}({B}^{\frac{d}{p}-1}_{p,1})\cap L^{1}_{T}({B}^{\frac{d}{p}+1}_{p,1})}}$ \textbf{and} $\|b^n_j-b^{\infty}_j\|_{L^{\infty}_{T}({B}^{\frac{d}{p}}_{p,1})}$ \textbf{for fixed $j$}.\\
Note that $(u^n_j,b^n_j)_{n\in\overline{\mathbb{N}}}$ satisfy
\begin{equation}\label{spnj1}
\left\{\begin{array}{lll}
u^{n}_{jt}+u^{n}_j\nabla u^{n}_j=b^{n}_j\nabla b^{n}_j+\frac{\nabla div}{-\Delta}(b^n\nabla b^n-u^n\nabla u^n),\\
b^{n}_{jt}-\Delta b^{n}_j+u^{n}_j\nabla b^{n}_j=b^{n}_j\nabla u^{n}_j ,\\
(u^n_0,b^n_0):=(\dot{S}_ju^n_0,\dot{S}_jb^n_0).
\end{array}\right.
\end{equation}
Applying Lemmas \ref{heat}--\ref{priori estimate} to \eqref{spnj1}, we have
\begin{align}\label{spnju}
\|u^n_j\|_{{B}^{\frac{d}{p}+2}_{p,1}}
&\leq \|{S}_ju^{n}_0\|_{{B}^{\frac{d}{p}+2}_{p,1}}
+\int_{0}^{t}\|u^n_j\|_{{B}^{\frac{d}{p}+1}_{p,1}}\|u^n_j\|_{{B}^{\frac{d}{p}+2}_{p,1}}
+\|b^n_j\|_{{B}^{\frac{d}{p}}_{p,1}}\|b^n_j\|_{{B}^{\frac{d}{p}+3}_{p,1}}{\ud}\tau\notag\\
&\leq 2^jC\eta\|{S}_ju^{n}_0\|_{{B}^{\frac{d}{p}+1}_{p,1}}
+\int_{0}^{t}12E_0C\|u^n_j\|_{{B}^{\frac{d}{p}+2}_{p,1}}{\ud}\tau
+12E_0C\|b^n_j\|_{L^1_t({B}^{\frac{d}{p}+3}_{p,1})}
\end{align}
and
\begin{align}\label{spnjb}
\|b^n_j\|_{{B}^{\frac{d}{p}+1}_{p,1}\cap L^1_T({B}^{\frac{d}{p}+3}_{p,1})}
&\leq \|{S}_jb^{n}_0\|_{{B}^{\frac{d}{p}+1}_{p,1}}
+C\int_{0}^{t}\|b^n_j\|_{{B}^{\frac{d}{p}+2}_{p,1}}\|u^n_j\|_{{B}^{\frac{d}{p}}_{p,1}}
+\|b^n_j\|_{{B}^{\frac{d}{p}}_{p,1}}\|u^n_j\|_{{B}^{\frac{d}{p}+2}_{p,1}}{\ud}\tau\notag\\
&\leq 2^j\|{S}_jb^{n}_0\|_{{B}^{\frac{d}{p}}_{p,1}}
+ 12CT^{\frac{1}{2}}E_0\|b^n_j\|_{L^2_T({B}^{\frac{d}{p}+2}_{p,1})}+12E_0C\int_{0}^{t}\|u^n_j\|_{{B}^{\frac{d}{p}+2}_{p,1}}{\ud}\tau,
\end{align}
where we used the fact that $\|\dot{S}_jg\|_{{B}^{\frac{d}{p}}_{p,1}}\leq C2^{m}\|\dot{S}_jg\|_{{B}^{\frac{d}{p}-m}_{p,1}},~m>0 .$

By virtue of \eqref{spnju}, \eqref{spnjb}$\times (48E_0C+1)$ and the Gronwall inequality, we obtain
\begin{align}\label{cd1}
\|b^n_j\|_{L^{\infty}_T({B}^{\frac{d}{p}+1}_{p,1})\cap L^1_T({B}^{\frac{d}{p}+3}_{p,1})}
+\|u^n_j\|_{{B}^{\frac{d}{p}+2}_{p,1}}
\leq 2^jC(\|b^{n}_0\|_{{B}^{\frac{d}{p}}_{p,1}}+\|u^{n}_0\|_{{B}^{\frac{d}{p}+1}_{p,1}})
\leq C_{E_0}2^j.
\end{align}

For fixed $j$, letting $\delta^n u=u^n_j-u^{\infty}_j$ and $\delta^n b=b^n_j-b^{\infty}_j$, then $(\delta^n u,\delta^n b)$ satifies
\begin{equation}
\left\{\begin{array}{lll}
(\delta^nu)_t+u^n_j\nabla(\delta^n u)+(\delta^n)u\nabla u^{\infty}_j+\nabla (P^n_j-P^{\infty}_j)=b^n_j\nabla (\delta^nb)+(\delta^nb)\nabla b^{\infty}_j,\\
(\delta^n b)_t-\Delta(\delta^n b)+u^n_j\nabla(\delta^n b)+(\delta^nu)\nabla b^{\infty}_j=b^n_j\nabla(\delta^n u)+(\delta^n b)\nabla u^{\infty}_j ,\\
(\delta^nu,\delta^nb)|_{t=0}=({S}_ju^n_0,{S}_jb^n_0).
\end{array}\right.
\end{equation}
According to \eqref{uniquegj1}, we see
\begin{align}\label{cd2}
&\|\delta^n b\|_{L^{\infty}_T ({B}^{\frac{d}{p}}_{p,1})\cap L^{1}_T ({B}^{\frac{d}{p}+2}_{p,1})}+\|\delta u\|_{L^{\infty}_T ({B}^{\frac{d}{p}+1}_{p,1})}\notag\\
\leq&
e^{A(T)}(\|\delta^n b_0\|_{{B}^{\frac{d}{p}}_{p,1}}+\|\delta^n u_0\|_{{B}^{\frac{d}{p}+1}_{p,1}}+\int_{0}^{T}\|u^{\infty}_j\|_{{B}^{\frac{d}{p}+2}_{p,1}}\|\delta^nu\|_{{B}^{\frac{d}{p}+1}_{p,1}}{\ud}\tau),
\end{align}
where $A(T):=C_{E_0}\int_{0}^{T}1+\|u^{\infty}_j\|_{{B}^{\frac{d}{p}+1}_{p,1}}+\|u^{n}_j\|_{{B}^{\frac{d}{p}+1}_{p,1}}+\|b^{n}_j\|_{{B}^{\frac{d}{p}+2}_{p,1}}+\|b^{\infty}_j\|_{{B}^{\frac{d}{p}+2}_{p,1}}{\ud}\tau\leq C_{E_0,T}$. The Gronwall's inequality and \eqref{cd1} together yield
\begin{align}
\|\delta^n b\|_{L^{\infty}_T ({B}^{\frac{d}{p}}_{p,1})\cap L^{1}_T ({B}^{\frac{d}{p}+2}_{p,1})}+\|\delta u\|_{L^{\infty}_T ({B}^{\frac{d}{p}+1}_{p,1})}\leq
 e^{C_{E_0,T}2^j}(\|\delta^n b_0\|_{{B}^{\frac{d}{p}}_{p,1}}+\|\delta^n u_0\|_{{B}^{\frac{d}{p}+1}_{p,1}}),
\end{align}
which implies that for any fixed $j$,
\begin{align}\label{sp8}
\|u^n_j-u^{\infty}_j\|_{L^{\infty}_{T}({B}^{\frac{d}{p}+1}_{p,1})}+\|b^n_j-b^{\infty}_j\|_{L^{\infty}_{T}({B}^{\frac{d}{p}}_{p,1})\cap L^1_{T}({B}^{\frac{d}{p}+2}_{p,1})} \rightarrow 0,\quad n\rightarrow\infty .
\end{align}

\textbf{2. Estimate }$\|b^n-b^n_j\|_{L^{\infty}_{T}({B}^{\frac{d}{p} }_{p,1})\cap L^1_{T}({B}^{\frac{d}{p}+2}_{p,1})}$ \textbf{for any $n\in\overline{\mathbb{N}}$} .

Letting $\delta_j u=u^n-u^{n}_j$ and $\delta_j b=b^n-b^{n}_j$, then we have
\begin{equation}\label{spp1}
\left\{\begin{array}{lll}
(\delta_ju)_t+u^n\nabla(\delta_j u)+(\delta_ju)\nabla u^n_j+\nabla (P^n-P^n_j)=b^n\nabla (\delta_jb)+(\delta_jb)\nabla b^n_j,\\
(\delta_j b)_t-\Delta(\delta_j b)+u^n\nabla(\delta_j b)+(\delta_ju)\nabla b^n_j=b^n\nabla(\delta_j u)+(\delta_j b)\nabla u^n_j ,\\
(\delta_ju_0,\delta_jb_0)|_{t=0}=\big((Id-S_j)u^n_0,(Id-S_j)b^n_0\big).
\end{array}\right.
\end{equation}
Thanks to \eqref{uniquegj2} in Proposition \ref{uniqueness}, we get
\begin{align}\label{uniquegj22}
\|\delta_j b\|_{L^{\infty}_T ({B}^{\frac{d}{p}-1}_{p,\infty})\cap L^{1}_T ({B}^{\frac{d}{p}+1}_{p,\infty})}+\|\delta_j u\|_{L^{\infty}_T ({B}^{\frac{d}{p}}_{p,\infty})}&\leq
e^{e^{A(T)}}(\|(Id-S_j)b_0\|_{{B}^{\frac{d}{p}-1}_{p,\infty}}+\|(Id-S_j)u_0\|_{{B}^{\frac{d}{p}}_{p,\infty}})\notag\\
&\leq C_{E_0,T}2^{-j}(\|(Id-S_j)b_0\|_{{B}^{\frac{d}{p} }_{p,\infty}}+\|(Id-S_j)u_0\|_{{B}^{\frac{d}{p}+1}_{p,\infty}})\notag\\
&\rightarrow 0,\quad j\rightarrow \infty.
\end{align}
Since $\|\delta_j u\|_{L^{\infty}_T ({B}^{\frac{d}{p}+1}_{p,1})}\leq C_{E_0} $, we deduce by the interpolation inequality that
\begin{align}\label{step2}
\|\delta_j u\|_{L^{\infty}_T ({B}^{\frac{d}{p}+1-\epsilon}_{p,1})}\rightarrow 0,\quad j\rightarrow \infty,\quad \forall \epsilon>0.
\end{align}

Next we estimate $\|\delta_jb\|_{L^{\infty}_{T}({B}^{\frac{d}{p}}_{p,1})\cap L^1_{T}({B}^{\frac{d}{p}+2}_{p,1})}$. Similarly, taking advantage of \eqref{uniquegj1} in Proposition \ref{uniqueness}, we find
\begin{align}\label{spp33}
\|\delta_j b\|_{L^{\infty}_{T}({B}^{\frac{d}{p}}_{p,1})\cap L^1_T({B}^{\frac{d}{p}+2}_{p,1})}
\leq C_{E_0}\Big(\|(Id-S_j)b_0\|_{{B}^{\frac{d}{p}}_{p,1}}+\int_{0}^{t}\|\delta_j u\|_{{B}^{\frac{d}{p}+1}_{p,1}}{\ud}\tau\Big),
\end{align}
which implies we need to estimate  $\|\delta_ju\|_{{B}^{\frac{d}{p}+1}_{p,1}}$ to obtain the continuous dependence of $(\delta_ju,\delta_jb)$.

\textbf{3. Estimate } $\|u^n-u^n_j\|_{L^{\infty}_{T}({B}^{\frac{d}{p}+1}_{p,1})}$ \textbf{for any $n\in\overline{\mathbb{N}}$} .

Define that $\Omega^{n}_{j}:={\rm curl} u^{n}_j$, $u^{n}_{\infty}:=u^{n}$, $b^{n}_{\infty}:=b^{n}$. Then $\Omega^n_j$ satisfies:
\begin{equation}\label{equ1-1-5}
  \left\{\begin{array}{l}
   \frac{{\ud}}{{\ud}t}\Omega^n_j+u^n_j\nabla \Omega^n_j=\Omega^n_j\nabla u^n_j+b^n_j\nabla {\rm curl}b^n_j-{\rm curl}b^n_j\nabla b^n_j,  \\
   \Omega^n_j(0,x)=\dot{S}_j({\rm curl}u^n_0).
  \end{array}\right.
\end{equation}
Let $\Omega^n_j:=w^n_j+z^n_j$ such that
\begin{equation}\label{spp5}
\left\{\begin{array}{lll}
\frac{{\ud}}{{\ud}t}w^n_j+u^n_j\nabla w^n_j=F^{\infty},\\
w^n_j|_{t=0}={\rm curl}u^n_0
\end{array}\right.
\end{equation}
and
\begin{equation}\label{spp6}
\left\{\begin{array}{lll}
\frac{{\ud}}{{\ud}t}z^n_j+u^n_j\nabla z^n_j=F^{j}-F^{\infty},\\
z^n_j|_{t=0}=(\dot{S}_j-Id){\rm curl}u^n_0,
\end{array}\right.
\end{equation}
where $ F^{j}:=\Omega^n_j\nabla u^n_j+b^n_j\nabla {\rm curl}b^n_j-{\rm curl}b^n_j\nabla b^n_j$
and $F^{\infty}:=\Omega^n_{\infty}\nabla u^n_{\infty}+b^n_{\infty}\nabla {\rm curl}b^n_{\infty}-{\rm curl}b^n_{\infty}\nabla b^n_{\infty}$.

Since $F^{\infty},F^{j}$ are bounded in $L^1_{T}({B}^{\frac{d}{p}}_{p,1})$, by use of Remark \ref{priori estimate1} and Theorem 3.19 in \cite{book}, we deduce that \eqref{spp5} and \eqref{spp6} have a unique solution $w^n_j,z^n_j\in C_{T}({B}^{\frac{d}{p}}_{p,1})$.

Our main idea is to verify that
$(w^n_j,z^n_j)\rightarrow (w^n_{\infty},0)\text{ in }{B}^{\frac{d}{p}}_{p,1}$ for any $n\in\overline{\mathbb{N}}$, which implies that
$ u^n_j\rightarrow u^n_{\infty}\text{ in }{B}^{\frac{d}{p}+1}_{p,1}.$
For this purpose, we divide the verification into the following three parts.

Firstly, we estimate $\|w^n_j-w^n_{\infty}\|_{L^{\infty}_{T}({B}^{\frac{d}{p}}_{p,1})}$. Taking advantage of Lemma \ref{Convergence} and \eqref{step2}, one can easily obtain
\begin{align}\label{spp14}
\|w^n_j-w^n_{\infty}\|_{L^{\infty}_{T}({B}^{\frac{d}{p}}_{p,1})}\rightarrow 0,\quad j\rightarrow\infty ,\quad\forall n\in \overline{\mathbb{N}}.
\end{align}

Next, we estimate $\|z^n_j\|_{L^{\infty}_{T}({B}^{\frac{d}{p}}_{p,1})}$. Since ${\rm div}({\rm curl}u)={\rm div}u=0 $, we see
\begin{align}\label{spp16}
&\int_{0}^{T}\|F^{j}-F^{\infty}\|_{{B}^{\frac{d}{p}}_{p,1}}{\ud}\tau \notag\\
&\leq \int_{0}^{T}\|(\Omega^n_j-\Omega^n_{\infty})\nabla u^n_j\|_{{B}^{\frac{d}{p}}_{p,1}}+\|\Omega^n_{\infty}\nabla (u^n_j-u^n_{\infty})\|_{{B}^{\frac{d}{p}}_{p,1}}
+\|(b^n_j-b^n_{\infty})\nabla {\rm curl}b^n_j\|_{{B}^{\frac{d}{p}}_{p,1}}\notag\\
&\quad +\|b^n_{\infty}\nabla ({\rm curl}b^n_j-{\rm curl}b^n_{\infty})\|_{{B}^{\frac{d}{p}}_{p,1}}
+\|({\rm curl}b^n_j-{\rm curl}b^n_{\infty})\nabla b^n_j\|_{{B}^{\frac{d}{p}}_{p,1}}+\|{\rm curl}b^n_{\infty}\nabla (b^n_j-b^n_{\infty})\|_{{B}^{\frac{d}{p}}_{p,1}}{\ud}\tau\notag\\
&\leq \int_{0}^{T}6E_0C(\|\Omega^n_j-\Omega^n_{\infty}\|_{{B}^{\frac{d}{p}}_{p,1}}+ \|u^n_j-u^n_{\infty}\|_{{B}^{\frac{d}{p}+1}_{p,1}})
+\|b^n_j\|_{{B}^{\frac{d}{p}+2}_{p,1}}\|b^n_j-b^n_{\infty}\|_{{B}^{\frac{d}{p}}_{p,1}}{\ud}\tau\notag\\
&\quad +6E_0C\|b^n_j-b^n_{\infty}\|_{L^1_T({B}^{\frac{d}{p}+2}_{p,1})}
+4aC\|b^n_j-b^n_{\infty}\|_{L^2_T({B}^{\frac{d}{p}+1}_{p,1})}
\notag\\
&\leq \int_{0}^{T}6E_0C(\|z^n_j\|_{{B}^{\frac{d}{p}}_{p,1}}+\|w^n_j-w^n_{\infty}\|_{{B}^{\frac{d}{p}}_{p,1}}+\|u^n_j-u^n_{\infty}\|_{{B}^{\frac{d}{p}}_{p,1}})
+\|b^n_j\|_{{B}^{\frac{d}{p}+2}_{p,1}}\|b^n_j-b^n_{\infty}\|_{{B}^{\frac{d}{p}}_{p,1}}{\ud}\tau\notag\\
&\quad +6E_0C\|b^n_j-b^n_{\infty}\|_{L^1_T({B}^{\frac{d}{p}+2}_{p,1})}
+4aC\|b^n_j-b^n_{\infty}\|_{L^2_T({B}^{\frac{d}{p}+1}_{p,1})}
\end{align}
where the last inequality is based on $\|\Omega^n_j-\Omega^n_{\infty}\|_{{B}^{\frac{d}{p}}_{p,1}}\leq \|w^n_j-w^n_{\infty}\|_{{B}^{\frac{d}{p}}_{p,1}}+\|z^n_j\|_{{B}^{\frac{d}{p}}_{p,1}}$. \eqref{spp6} and \eqref{spp16} ensure that
\begin{align}\label{hrecallspp33}
\|z^n_j\|_{{B}^{\frac{d}{p}}_{p,1}}
\leq&\|(Id-\dot{S}_j){\rm curl} u^n_0\|_{{B}^{\frac{d}{p}}_{p,1}}+6E_0C\|b^n_j-b^n_{\infty}\|_{L^1_T{B}^{\frac{d}{p}+2}_{p,1}}
+4aC\|b^n_j-b^n_{\infty}\|_{L^2_T{B}^{\frac{d}{p}+1}_{p,1}}\notag\\
&+\int_{0}^{T}6E_0C(\|z^n_j\|_{{B}^{\frac{d}{p}}_{p,1}}+\|w^n_j-w^n_{\infty}\|_{{B}^{\frac{d}{p}}_{p,1}}+\|u^n_j-u^n_{\infty}\|_{{B}^{\frac{d}{p}}_{p,1}})
+\|b^n_j\|_{{B}^{\frac{d}{p}+2}_{p,1}}\|b^n_j-b^n_{\infty}\|_{{B}^{\frac{d}{p}}_{p,1}}{\ud}\tau.
\end{align}
By \eqref{spp33}, we find
\begin{align}\label{recallspp33}
&\|b^n_j-b^n_{\infty}\|_{L^{\infty}({B}^{\frac{d}{p}}_{p,1})\cap L^2_T{B}^{\frac{d}{p}+1}_{p,1}\cap L^1_T{B}^{\frac{d}{p}+2}_{p,1}}\notag\\
\leq& C_{E_0}(\|(Id-S_j)b_0\|_{{B}^{\frac{d}{p}}_{p,\infty}}+\int_{0}^{t}\|u^n_j-u^n_{\infty}\|_{{B}^{\frac{d}{p}+1}_{p,1}}ds)\notag\\
\leq& C'_{E_0}(\|(Id-S_j)b_0\|_{{B}^{\frac{d}{p}}_{p,\infty}}+\int_{0}^{t}\|u^n_j-u^n_{\infty}\|_{{B}^{\frac{d}{p}}_{p,1}}+\|w^n_j-w^n_{\infty}\|_{{B}^{\frac{d}{p}}_{p,1}}+\|z^n_j\|_{{B}^{\frac{d}{p}}_{p,1}}{\ud}\tau).
\end{align}
Combining \eqref{hrecallspp33} with \eqref{recallspp33}$\times (48E_0C+8aC+1)$, and then applying the Gronwall inequality, we have
\begin{align}\label{spp18}
&\|z^n_j\|_{L^{\infty}_{t}({B}^{\frac{d}{p}}_{p,1})}+\|b^n_j-b^n_{\infty}\|_{L^{\infty}({B}^{\frac{d}{p}}_{p,1})\cap L^2_T{B}^{\frac{d}{p}+1}_{p,1}\cap L^1_T{B}^{\frac{d}{p}+2}_{p,1}}\notag\\
\leq &  C_{E_0,T}(\|(Id-S_j)u^n_0\|_{{B}^{\frac{d}{p}+1}_{p,1}}+\|(Id-S_j)b^n_0\|_{{B}^{\frac{d}{p}}_{p,1}}+
\|w^n_j-w^n_{\infty}\|_{L^{\infty}({B}^{\frac{d}{p}}_{p,1})}+\|u^n_j-u^n_{\infty}\|_{L^{\infty}({B}^{\frac{d}{p}}_{p,1})}) \notag\\
\rightarrow & 0,\quad j\rightarrow\infty,\quad\forall n\in \overline{\mathbb{N}},
\end{align}
where the last inequality holds by \eqref{spp14} and \eqref{step2}.

Finally, from \eqref{spp14} and \eqref{spp18}, we can obtain
\begin{align}\label{spp19}
\|b^n_j-b^n_{\infty}\|_{L^{\infty}_{T}({B}^{\frac{d}{p}}_{p,1})\cap L^1_T({B}^{\frac{d}{p}+2}_{p,1})}\rightarrow 0\quad and\quad
\|u^n_j-u^n_{\infty}\|_{L^{\infty}_{T}({B}^{\frac{d}{p}+1}_{p,1})}\rightarrow 0,\quad j\rightarrow\infty,\quad \forall n\in \overline{\mathbb{N}}.
\end{align}

Thus, by \textbf{1-3}
we can prove the continuous dependence. In fact, combining \eqref{sp8} and \eqref{spp19}, we see
\begin{align*}
\|u^n-u^{\infty}\|_{L^{\infty}_{T}({B}^{\frac{d}{p}-1}_{p,1})\cap L^1_{T}({B}^{\frac{d}{p}+1}_{p,1})}+\|b^n-b^{\infty}\|_{L^{\infty}_{T}({B}^{\frac{d}{p}}_{p,1})}
\rightarrow 0,\quad n\rightarrow +\infty,
\end{align*}
which implies the continuous dependence.
%As $(u^n,b^n)$ is uniformly bounded in ${L^{\infty}_{T_n}{B}^{\frac{d}{p}-1}_{p,1}\cap L^{1}_{T_n}{B}^{\frac{d}{p}+1}_{p,1}}\times L^{\infty}_{T_n}{B}^{\frac{d}{p}}_{p,1}$,
%we can extent the maximal time of continuous dependence to $T_n$(at least to $T$) within a finite number of steps.
\end{proof}

Therefore, Combining the proof of local existence, uniqueness and continuous dependence in Section 3 and Section 4, we obtain Theorem \ref{theorem}.
%\begin{lemm}\label{blowup}
%Let $(u,b)\in (C_T({B}^{s}_{p,r}),~C_T({B}^{s-1}_{p,r})\cap L^1_T({B}^{s+1}_{p,r}))$ be a solution of \eqref{sp00} with $s>1+\frac{2}{p},~1\leq p,r\leq \infty$. If
%$$\int_{0}^{T}\|u\|_{B^1_{\infty,1}}ds<\infty,$$
%then the solution $(u,b)$ exists globally.
%\end{lemm}
\section{Global existence}
\textbf{The proof of Theorem \ref{theoremglobal1}:}
\begin{proof}
%Without loss of generality, one can only consider the low regularity cases such as $(u_0,b_0)\in B^s_{p,r}\times B^{s-1}_{p,r}$ with $1+\frac{d}{p}<s<2+\frac{d}{p},~1\leq p,r\leq\infty$, since the high regularity cases are more easier.
We use the bootstrap argument to prove this Theorem. Without loss of generality, we only consider the critical case: $s=1+\frac{2}{p}$. Let $T^*$ be the maximal existence time of the solution. Assume that for any $t\leq T< T^*$,
\begin{align}\label{global00}
\|b\|_{L^{\infty}_T(B^{0}_{\infty,1})\cap L^{1}_T(B^{2}_{\infty,1})}\leq 4.
\end{align}
Let $\epsilon$ be a sufficient small positive constant such that
\begin{align}\label{globaljie1}
\|u_0\|_{B^1_{\infty,1}(\mathbb{S}^2)}+\|b_0\|_{B^0_{\infty,1}(\mathbb{S}^2)}\leq\epsilon\leq \frac{c}{16(1+C^{20})},
\end{align}
where $0<6c<1$ and $20<C$ are some fixed constants and will be determined later. The proof can be divided into 4 parts:

(1) First, we give the estimation of $\|b(t)\|_{L^2}$.\\
It's trivial to verify that
$$\frac{1}{2}(\|u\|^2_{L^2}+\|b\|^2_{L^2})+\int_{0}^{T}\|b\|^2_{\dot{H}^1}{\ud}\tau =\frac{1}{2}(\|u_0\|^2_{L^2}+\|b_0\|^2_{L^2}).$$

Let $w={\rm curl}u$ with $d=2$, we have
\begin{align}\label{global6}
w_t+u\nabla w=b\nabla {\rm curl}b.
\end{align}
By \eqref{global00}, we deduce that
\begin{align}\label{globalbuchong1}
\|w\|_{L^{\infty}_T(L^{\infty})}\leq \|u_0\|_{B^{1}_{\infty,1}}+C\|b\|_{L^{\infty}_T(B^{0}_{\infty,1})}\|b\|_{ L^{1}_T(B^{2}_{\infty,1})}\leq C(\epsilon+16)\leq C^2.
\end{align}

Since $\int_{\mathbb{S}}b_0{\ud}x=0$, one can easily deduce that $\int_{\mathbb{S}}b{\ud}x=0$, which means that $\Delta_{-1}b=0$ and $\|b\|_{L^2(\mathbb{S})}\leq C\|b\|_{\dot{H}^1(\mathbb{S})}$. Taking $L^2$ inner product with $b$ to the second equation of \eqref{sp00}, we have
\begin{align}\label{global1}
\frac{{\ud}}{{\ud}t}\|b\|^2_{L^2}+\|b\|^2_{\dot{H}^1}\leq & C\|u\|_{L^{\infty}}\|b\|_{L^2}\|b\|_{\dot{H}^1}\notag\\
\leq & C(\|w\|^{\frac{1}{2}}_{L^{\infty}}\|u_0\|^{\frac{1}{2}}_{L^2})\|b\|_{L^2}\|b\|_{\dot{H}^1}\notag\\
\leq & C^2\epsilon^{\frac{1}{2}}\|b\|_{L^2}\|b\|_{\dot{H}^1}\notag\\
\leq & \frac{1}{2}\|b\|^2_{\dot{H}^1}
\end{align}
Hence, there exists a positive constant $0<6c<1$ such that
\begin{align}\label{global2}
\|b\|_{L^2}\leq C\|b_0\|_{L^2}e^{-6ct}\leq C\epsilon e^{-6ct}.
\end{align}

(2) Then, we give the estimation of $\|b(t)\|_{L^{\infty}}.$

Using the fact that $\|\nabla^ke^{t\Delta}f\|_{L^{q}}\leq Ct^{-\frac{k}{2}-\frac{1}{p}+\frac{1}{q}}\|f\|_{L^{p}}~(1\leq p\leq q\leq \infty)$ for $\int_{\mathbb{S}^2}fdx=0$ in \cite{weishen}, we have
\begin{align}\label{xin3}
\|b\|_{L^{\infty}}&\leq C\|b_0\|_{L^{\infty}}e^{-t}+\int_{0}^{t}(t-s)^{-\frac{2}{3}}e^{-(t-s)}\|b\nabla u\|_{L^{\frac{3}{2}}}ds+\int_{0}^{t}e^{-(t-s)}\|u\nabla b\|_{L^{\infty}}ds\notag\\
&\leq C\epsilon e^{- 6ct}+\int_{0}^{t}(t-s)^{-\frac{2}{3}}e^{-(t-s)}\|b\|_{L^{2}}\|w\|_{L^{6}}ds+\int_{0}^{t}e^{-(t-s)}\|b\|_{L^2}^{\frac{1}{3}}\|b\|_{B^2_{\infty,1}}^{\frac{2}{3}}\|u\|_{L^{\infty}}ds\notag\\
&\leq C\epsilon e^{- 6ct}+C^3\epsilon\int_{0}^{t}(t-s)^{-\frac{2}{3}}e^{-2c(t-s)} e^{-2cs}ds+C^3\epsilon^{\frac{1}{3}}\int_{0}^{t}e^{-2c(t-s)}e^{-2cs}\|b\|_{B^2_{\infty,1}}^{\frac{2}{3}}ds\notag\\
&\leq  C\epsilon e^{- 6ct}+C^3\epsilon t^{\frac{1}{3}}e^{-2 ct}+C^3\epsilon^{\frac{1}{3}} t^{\frac{1}{3}}e^{-2 ct}\|b\|_{L^1_t(B^2_{\infty,1})}^{\frac{2}{3}}\notag\\
&\leq  C^4 \epsilon^{\frac{1}{3}} e^{-ct}.
\end{align}
That is
\begin{align}\label{global3.5}
\|b\|_{L^{\infty}}\leq C^4\epsilon^{\frac{1}{3}} e^{-ct}.
\end{align}

(3) Next, we give the estimation of $\|w(t)\|_{B^{0}_{\infty,1}} $.\\
Applying Lemma \ref{priori estimate} to \eqref{global6}, we have
\begin{align}\label{global7}
\|w\|_{L^{\infty}_{T}(B^{0}_{\infty,1})}\leq & C(\|w_0\|_{B^{0}_{\infty,1}}+\|b\|_{L^{\infty}_T(L^{\infty})}\|b\|_{ L^{1}_T(B^{2}_{\infty,1})})(1+\int_{0}^{T}\|w\|_{L^{\infty}_{T}(B^{0}_{\infty,1})}ds)\notag\\
\leq & C(\epsilon+ 4C^4\epsilon^{\frac{1}{3}} )(1+\int_{0}^{T}\|w\|_{L^{\infty}_{T}(B^{0}_{\infty,1})}ds)\notag\\
\leq & c_{small}e^{c_{small}t}
\end{align}
where the last inequality holds by the Gronwall's inequality, and $c_{small}:=5C^5\epsilon^{\frac{1}{3}}\leq \frac{1}{2} c<\frac{1}{2}$~(see \eqref{globaljie1}).

(4) Finally, thanks to \eqref{global7} and \eqref{global3.5}, we conclude that
\begin{align}
\|b\|_{L^{\infty}_{T}(B^0_{\infty,1})\cap L^{1}_{T}(B^2_{\infty,1})}\leq & \|b_0\|_{B^0_{\infty,1}}+C\|-u\nabla b+b\nabla u\|_{L^{1}_{t}(B^{0}_{\infty,1})}\notag\\
\leq & \|b_0\|_{B^0_{\infty,1}}+C\|u\|_{L^{\infty}_{T}(L^{2})}\|b\|_{L^{1}_{T}(B^{2}_{\infty,1})}+C\|u\|_{L^{\infty}_{T}(B^{1}_{\infty,1})}\|b\|_{L^{\infty}_{T}(L^{\infty})}\notag\\
\leq & \|b_0\|_{B^0_{\infty,1}}+\frac{1}{4}\|b\|_{L^{1}_{T}(B^2_{\infty,1})}+C\int_{0}^{T}(\|u\|_{L^{\infty}_{T}(L^2)}+\|w\|_{B^{0}_{\infty,1}})\|b\|_{L^{\infty}}ds,\notag\\
\leq & \|b_0\|_{B^0_{\infty,1}}+\frac{1}{4}\|b\|_{L^{1}_{T}(B^2_{\infty,1})}+C\int_{0}^{T}(\epsilon+c_{small}e^{c_{small}s})C^4\epsilon^{\frac{1}{3}}e^{-cs}ds ,\notag\\
\leq & \frac{4}{3}(\epsilon+C^5\epsilon^{\frac{1}{3}}\int_{0}^{T}e^{c_{small}s}e^{- cs}ds) \notag\\
\leq & \frac{4}{3}*2 .
\end{align}
where the last inequality holds by $c_{small}\leq\frac{1}{2} c<\frac{1}{2}$ and $\epsilon\leq  \frac{c}{16(1+C^{20})}$ ~(see \eqref{globaljie1}).

So far, by (1)-(4) and the bootstrap argument we have proved that
\begin{align}\label{global18}
 \|b\|_{L^{\infty}_{t}(B^0_{\infty,1})\cap L^{1}_{t}(B^2_{\infty,1})}\leq \frac{8}{3}<4,\quad \forall t\in[0,T^*).
\end{align}

Then, one can obtain the global existence of $(u,b)$ in $ C([0,\infty); {B}^{s}_{p,r})\times \Big(C([0,\infty);{B}^{s-1}_{p,r})\cap L^1\big([0,\infty);{B}^{s+1}_{p,r}\big)\Big)$ easily, since $\|b\|_{L^{\infty}_{t}(B^0_{\infty,1})\cap L^{1}_{t}(B^2_{\infty,1})}\leq 4$ can be the blow-up criteria for \eqref{sp00}.
Indeed, similar to the computations in (1)-(3), one has
\begin{align}\label{global19}
\|w(t)\|_{B^{0}_{\infty,1}}\leq c_{small}e^{c_{small}t} \quad and\quad \|b(t)\|_{L^{\infty}}\leq C^4\epsilon^{\frac{1}{3}} e^{-ct},~~~\forall t\in[0,T^*).
\end{align}
Moreover, combining \eqref{globaljie1} and \eqref{global19}, we deduce that
\begin{align}\label{globalbuchong3}
\|w(t)\|_{L^{\infty}}\leq \|u_0\|_{B^1_{\infty,1}}+C\|b\|_{L^{\infty}_t(L^{\infty})}\|b\|_{ L^{1}_t(B^{2}_{\infty,1})}\leq C(\epsilon+ 4C^4\epsilon^{\frac{1}{3}} )\leq \frac{1}{4C}.
\end{align}

%Then, it's easy to obtain the global existence of $(u,b)$ in $ C([0,\infty); {B}^{s}_{p,r})\times \Big(C([0,\infty);{B}^{s-1}_{p,r})\cap L^1\big([0,\infty);{B}^{s+1}_{p,r}\big)\Big)$. Since $s>1+\frac{2}{p}$, this represents a non-critical case and so we should only consider the case when $r=1,~1\leq p\leq \infty$.
Applying Lemma \ref{heat}--\ref{priori estimate} to \eqref{sp00}, using the fact that $\|u\|_{B^{s}_{p,1}}\leq C\big(\|u\|_{L^2}+\|w\|_{B^{s-1}_{p,1}}\big)$ in periodic case and \eqref{globaljie1}, we obtain
\begin{align}\label{global20}
\|u\|_{L^{\infty}_{t}({B^{1+\frac{2}{p}}_{p,1}})}\leq & \|u_0\|_{{B^{1+\frac{2}{p}}_{p,1}}}
+C\int_{0}^{t}\|u\|_{B^1_{\infty,1}}\|u\|_{B^{1+\frac{2}{p}}_{p,1}}{\ud}\tau+\frac{1}{4}\|b\|_{L^1_t(B^{2+\frac{2}{p}}_{p,1})}
\end{align}
and
\begin{align}\label{global21}
\|b\|_{L^{\infty}_t(B^{\frac{2}{p}}_{p,1})\cap L^{1}_t(B^{2+\frac{2}{p}}_{p,1})}\leq & \|b_0\|_{B^{\frac{2}{p}}_{p,1}}
+C\int_{0}^{t}\|u\|_{L^{\infty}}\|b\|_{B^{2+\frac{2}{p}}_{p,1}}+\|b\|_{L^{\infty}}\|u\|_{B^{1+\frac{2}{p}}_{p,1}}{\ud}\tau\notag\\
\leq &  \|b_0\|_{B^{\frac{2}{p}}_{p,1}}
+C\int_{0}^{t}(\|u\|_{L^{2}}+\|w\|_{L^{\infty}})\|b\|_{B^{2+\frac{2}{p}}_{p,1}}{\ud}\tau+ \frac{1}{4}\|u\|_{L^{\infty}_t(B^{1+\frac{2}{p}}_{p,1})}\notag\\
\leq & \|b_0\|_{B^{\frac{2}{p}}_{p,1}}
+ \frac{1}{2}\|b\|_{L^1_t(B^{2+\frac{2}{p}}_{p,1})}+ \frac{1}{4}\|u\|_{L^{\infty}_t(B^{1+\frac{2}{p}}_{p,1})}.
\end{align}
Combining \eqref{global20}--\eqref{global21} and  Gronwall's inequality, we see
\begin{align*}
\|u\|_{L^{\infty}_t({B^{1+\frac{2}{p}}_{p,1}})}+\|b\|_{L^{\infty}_t(B^{\frac{2}{p}}_{p,1})\cap L^{1}_t(B^{2+\frac{2}{p}}_{p,1})}\leq & C(\|u_0\|_{{B^{1+\frac{2}{p}}_{p,1}}}+\|b_0\|_{B^{\frac{2}{p}}_{p,1}}+
\int_{0}^{t}\|u\|_{{B^{1}_{\infty,1}}}\|u\|_{L^{\infty}_{\tau}(B^{1+\frac{2}{p}}_{p,1})}{\ud}\tau)\\
\leq & C(\|u_0\|_{{B^{1+\frac{2}{p}}_{p,1}}}+\|b_0\|_{B^{\frac{2}{p}}_{p,1}})e^{\int_{0}^{t}\|u\|_{{B^{1}_{\infty,1}}}{\ud}\tau}\\
\leq & Ce^{e^{ Ct}},\quad \forall t\in [0,T^*).
\end{align*}
This implies $T^*=\infty$.
\end{proof}

\noindent\textbf{Acknowledgements.} This work was partially supported by National Natural Science Foundation
of China [grant number 11671407 and 11701586], the Macao Science and Technology Development Fund (grant number 0091/2018/A3), Guangdong Special Support Program (grant number 8-2015), and the key project of NSF of  Guangdong province (grant number 2016A030311004).

\phantomsection
\addcontentsline{toc}{section}{\refname}

\bibliographystyle{abbrv} %plain ,%alpha, %abbrv
\bibliography{Ye-Yinref}

\end{document}